\documentclass{amsart}
\usepackage{graphicx}
\usepackage{amsmath,amssymb,amsfonts,amsthm}
\usepackage{amstext}
\usepackage{longtable}
\usepackage[english]{babel}
\newcommand{\note}{\noindent {\bf Notation. }}

\newcommand{\ws}{\hspace{4pt}}
\newtheorem{theorem}{Theorem}

\newtheorem{proposition}{Proposition}

\newtheorem{lemma}{Lemma}

\newtheorem{defi}{Definition}
\usepackage[]{amssymb,amsopn}
\usepackage[]{amsmath}
\usepackage[]{eucal}
\usepackage[]{latexsym}
\usepackage{color}
\usepackage{leftindex}

\makeatletter
\@namedef{subjclassname@2020}{%
  \textup{2020} Mathematics Subject Classification}
\makeatother

\begin{document}

\title[Chromatic expansion with Bessel operator]{Chromatic expansion with Bessel operator of fractional order}
\author{M. Chegaar and \'A. P. Horv\'ath}

\thanks{}

\begin{abstract}
This paper develops a Bessel-chromatic expansion framework associated with fractional powers of the Bessel-Laplace operator. The construction combines methods of weighted polynomial approximation and of fractional differential operators.

Using the spectral representation of $(-\Delta_a)^{\frac{1}{2}}$, we define Bessel-chromatic derivatives and apply them
to weighted spherical means both at a general point and at the origin.
Different classes of weights on finite and infinite intervals are considered, with particular attention to cases where the inverse Hankel transform is explicit.
The convergence of the expansions is studied through Ces\`aro and de la Vall\'ee Poussin means. In the bandlimited case, the method gives reconstruction
formulas for weighted spherical means and, under suitable assumptions, recovery formulas for the original function.
Numerical examples illustrate the decay of the Bessel-chromatic coefficients and the accuracy of the corresponding reconstructions.
\end{abstract}
\maketitle

\medskip

\section{Introduction}

Many problems in analysis and applied mathematics involve data whose natural geometry is not Cartesian, but radial, cylindrical, or spherical. This type of
structure appears in radial signal processing, numerical differentiation of oscillatory data, wave propagation in symmetric media. With an external field the ordinary Fourier transform is not always the most convenient spectral tool. The Hankel transform, together with Bessel-type differential operators, is often more suitable because it reflects the geometry of the problem more directly; see, for example, \cite{ss,p,z1}.

A classical tool for reconstructing bandlimited functions is the Whittaker-Shannon-Kotel'nikov sampling theorem. In its simplest form, under
the usual normalization and bandlimiting assumptions, it gives
$$f(t)=\sum_{n\in\mathbb{Z}} f(n)\frac{\sin \pi(t-n)}{\pi(t-n)}.$$
This formula is fundamental in signal processing. However, in numerical applications it may require many samples, and the convergence can be slow when
only finite information is available. Taylor expansion has a different advantage: it uses local differential information at one point. But this is
also its limitation, since many derivatives may be needed to approximate a function on an interval. Chromatic expansions, introduced by
Ignjatovi\'c \cite{i0}, offer a useful middle ground. They combine differential information in the physical variable with an orthogonal expansion
in the spectral variable. Further developments and applications of chromatic derivatives and chromatic expansions can be found in \cite{i1,i2,i3,i4,iz,z2,z3,h1}.

The main idea can be seen first in the Fourier setting. If $p$ is a polynomial and the inverse Fourier transform is available, then formally
$$p\left(-i\frac{\partial}{\partial x}\right)f(x)=c\int_\mathbb{R}p(y)\hat{f}e^{ixy}dy.$$
That is if $\mathrm{supp}\hat{f}\subset [a,b]$, and $\{p_k\}$ are the orthonormal polynomials on $[a,b]$; $p_k\left(-i\frac{\partial}{\partial x}\right)f(0)=cc_k(\hat{f})$, i.e. the Fourier coefficients of $\hat{f}$ with respect to $\{p_k\}$. Then the inverse Fourier transform of the series gives the chromatic expansion of $f$. Chromatic expansions have proven to be useful tool for rapid recovery of functions as well as for numerical differentiation, see e.g. \cite{i1,i3}. Investigation of the underlying spaces is also important, see e.g. \cite{i2}. Multidimensional extension can be found e.g. in \cite{iz}, and the problem is discussed in case of generalized functions in \cite{z3}. The questions of different weighted cases and different integral transforms naturally arose and were examined in numerous papers, see e.g. \cite{z2,h1}. The relationship of wavelets and chromatic expansions is discussed in \cite{h1}. For the summary of recent results, see \cite{i4}.

In this paper we use the Bessel differential operator, which in one dimension is
$$B_{\alpha}:=\frac{\partial^2}{\partial x^2}+\frac{2\alpha+1}{x}\frac{\partial}{\partial x},$$
and the corresponding Hankel transform, $\mathcal{H}_\alpha$. Since
$$\mathcal{H}_\alpha((-B_{\alpha}f))(y)=y^2\mathcal{H}_\alpha(f)(y),$$
in Bessel setup one can choose the entire Bessel functions, $j_\alpha(\sqrt{\lambda} x)$ as eigenfunction of the differential operator and so the kernel of the integral operator, cf. e.g. \cite{z1}, or taking the square-root of the $-B_{\alpha}$ operator, one can arrive to a situation which is very similar to the sketched one. Here we choose the second solution. In several variables we work on the positive octant \(\mathbb{R}^d_+\) with
the Bessel-Laplace operator
$$\Delta_a=\sum_{i=1}^d B_{\alpha_i}, \qquad a_i=2\alpha_i+1,$$
and with the corresponding $d$-dimensional Hankel transform. This choice is motivated by the intention to provide estimates of the of convergence rate of chromatic expansions.

At first we mention, that similarly to the Fourier case, a Shannon-type sampling formula can be derived (see e.g. \cite{k} and \cite{j}), but it possesses  rather poor convergence properties. We give a brief description of it in the Appendix.\\
On the other hand, since Bessel and Hankel operators act on functions defined on the positive octant, the functions under study can be extended evenly, and  the investigation can be transferred to a symmetric interval with symmetric weight. Our estimations on the rate of convergence are based on the smoothness properties of the functions on the image side. Since smoothness can be lost during even extension, this method generally results in weaker convergence.

Below we define spherical Bessel chromatic expansion on the positive octant of the $d$-dimensional space. It approximates the weighted spherical mean of a function. In some cases the original function can be recovered as well. To define chromatic derivatives and series we need some facts about Bessel translation, fractional Bessel operators and Hankel transform. These basic facts are contained in the second section.

The main goal of the investigation is to estimate the speed of convergence. To this, we use some results of weighted polynomial approximation. Although, as mentioned in the remarks at the end of the sections, our method and theorems can be extended to general classes of weights, we focus on weights with explicit inverse Hankel transform.

 It is known that the weighted spherical mean is the transmutation operator intertwining the $d$-dimensional and $1$-dimensional Bessel-Laplace operators, see \cite{ss}. In the third section we define spherical Bessel chromatic derivatives by the extension of this theorem to fractional operators. In the unweighted bandlimited case the original function can also be recovered from spherical expansion.

 In the fourth section, considering the weighted spherical mean at zero, we simplify the method. Although the possibility of global reconstruction is lost, the method has wide applicability as in the unweighted (Laplace) case, see e.g. \cite{cdps}.\\
 We mention here, that the $d$-dimensional Laplace and Bessel-Laplace operators are closely related, see e.g. \cite{ch1}. Thus, the results of this section can be transferred to the square root of the Laplace operator and the Fourier transform. Of course, in this case, it is more natural to consider $-i\frac{\partial}{\partial x}$ instead of $(-\Delta)^{\frac{1}{2}}$.

 A short fifth section has also been added, which contains the method of approximation of functions with even entire functions of exponential type $\nu$. Here we gave an example to the joint application of the different approximations.

 Finally, the Appendix, besides Shannon sampling, contains some numerical illustrations of Bessel chromatic expansion.

\section{Notation, Preliminaries}

\subsection{Spaces} We introduce the next weighted spaces. On $\mathbb{R}^{d}_{+}:=\{(x_1, \dots ,x_d): x_i>0 \ws i=1, \dots ,d\}$ we introduce the $L^p_a(\mathbb{R}^{d}_{+})$ space, $1\le p<\infty$ with norm
\begin{equation}\label{La}\|f\|_{p,a}^p=\int_{\mathbb{R}^d_{+}}|f(x)|^px^adx,\end{equation}
where $a=(a_1, \dots , a_d)$, $a_i>0$, $i=1, \dots , d$; $x^a=\prod_{i=1}^dx_i^{a_i}$. For $p=\infty$ we use the standard $L^\infty(\mathbb{R}^{d}_{+})$ space.

We also need the function space $\mathcal{S}_{e}$, the set of even functions from the Schwartz class, i.e.
\begin{equation}\label{S}\mathcal{S}_{e}:=\left\{f\in \mathcal{S} : \frac{\partial^{2m+1}}{\partial x_i^{2m+1}}f\large{|}_{x_i=0}=0, \ws \forall m \in \mathbb{N}, \ws i=1, \dots , d\right\}.\end{equation}
$\mathcal{S}_{e}'$ is the dual space of $\mathcal{S}_{e}$. First, as usual, we introduce the concepts for functions from $\mathcal{S}_{e}$, then using that $\mathcal{S}_e$ is dense in $L^p_a$ $1\le p <\infty$, we can extend them to $L^p_a$ with an appropriate $p$. We introduce the notation
$$\langle f,g\rangle_a=\int_{\mathbb{R}^{d}_{+}}f(x)g(x)x^adx,$$
and we denote in the same way the effect of an $f\in \mathcal{S}_{e}'$ on a $g\in \mathcal{S}_{e}$.

\note

$H\subset \mathbb{R}^d$. Then $H_+$ stands for $H\cap \mathbb{R}^d_+$.

\subsection{Bessel function and translation}
The entire Bessel functions are
\begin{equation}\label{B}j_\alpha(z)=\Gamma(\alpha+1)\left(\frac{2}{z}\right)^\alpha J_\alpha(z)=\sum_{k=0}^\infty\frac{(-1)^k\Gamma(\alpha+1)}{\Gamma(k+1)\Gamma(k+\alpha+1)}\left(\frac{z}{2}\right)^{2k},\end{equation}
where $\alpha>-\frac{1}{2}$.
We introduce the following abbreviation, cf. \cite{ss}.
$$\mathbb{j}_a(x,\xi):=\prod_{i=1}^dj_{\alpha_i}(x_i\xi_i).$$

The entire Bessel functions, $j_\alpha(\lambda z)$ given by \eqref{B} are the eigenfunctions of $B_\alpha$, i.e.
$$B_\alpha j_\alpha(\lambda z)=-\lambda^2j_\alpha(\lambda z).$$
The other important property of the entire Bessel functions is that their value at zero is one, independently of the parameter. Moreover we have
\begin{equation}\label{j1}\|j_\alpha\|_{\infty, \mathbb{R}_+}=j_\alpha(0)=1.\end{equation}
Subsequently we assume that the parameter of the Bessel functions, $\alpha>-\frac{1}{2}$. We note that $j_{-\frac{1}{2}}(x)=\cos x$. Below we need the derivation formula:
\begin{equation}\label{bder}j_\alpha'(z)=-\frac{1}{2(\alpha+1)}zj_{\alpha+1}(z),\end{equation}
(see \cite{be})

Subsequently we need the Bessel translation. The Bessel translation of a function $f$ (see e.g. \cite{le}, \cite{p}, \cite{ss}) is
$$T_{a}^tf(x)=T_{a_d}^{t_d}\dots T_{a_1}^{t_1}f(x_1,\dots ,x_d),$$
where
\begin{equation}\label{tra1}T_{a_i}^{t_i}f(x_1,\dots ,x_d)$$ $$=\frac{\Gamma(\alpha_i+1)}{\sqrt{\pi}\Gamma\left(\alpha_i+\frac{1}{2}\right)}\int_0^\pi f(x_1,\dots,\sqrt{x_i^2+t_i^2 -2x_it_i\cos\vartheta_i}, x_{i+1}, \dots, x_d)\sin\vartheta^{2\alpha_i}d\vartheta_i.\end{equation}

The immediate consequences of the definition of the Bessel translation, cf. \eqref{tra1} are that  $T_a$ is positive operator and
\begin{equation}\label{trid}T^0_af(x)=f(x); \ws \ws \ws T_{a}^tf(x)=T_{a}^xf(t).\end{equation}
Bessel translation is bounded in $L^p_a$, i.e. for all $a$ ($a_i> 0$),
\begin{equation}\label{t1}\|T_{a}^tf(x)\|_{p,a}\le \|f\|_{p,a}, \ws \ws 1\le p \le\infty,\end{equation}
see e.g. \cite{le}. If it does not cause any misunderstanding, we omit index $a$.

For sake of simplicity formulating in one dimension, for an $f\in C^2(\mathbb{R}_+)$, $T_\alpha^tf(x)=u(t,x)$ is the solution of the Cauchy problem
$$B_{\alpha,t}u(x,t)=B_{\alpha,x}u(x,t); \ws \ws u(x,0)=f(x), \ws \ws \partial_tu(x,0)=0.$$
Thus for $t,x>0$ we have
$$T_{\alpha}^tj_\alpha(\lambda x)=j_\alpha(\lambda t)j_\alpha(\lambda x).$$
Let $f\in  \mathcal{S}_e'$, $g\in \mathcal{S}_e$. As in the unweighted case translation and differential operator are defined
\begin{equation}\label{tracs}\langle T^tf,g\rangle_a= \langle f,T^tg\rangle_a\end{equation}
see e.g. \cite[(2.18)]{p},
\begin{equation}\label{Bcser}B_\alpha(T^tf)=T^t(B_\alpha f),\end{equation}
see e.g. \cite[(2.20)]{p}.

The generalized convolution indicated by  Bessel translation is
\begin{equation}\label{conv1}f*_a g =\int_{\mathbb{R}^d_+}T_{a,x}^tf(x)g(x)x^adx.\end{equation}

Throughout the paper we use the notation $x=r\theta$, where $r=|x|=\sqrt{\sum_{i=1}^d x_i^2}$ and $|\theta|=1$. The spherical mean of $f$ at $x$ is a function of $r$ is as follows.
\begin{equation}\label{sphm}M_{sph,a}(f)(x,t):=\frac{1}{|S^d_{1+}|_a}\int_{S^d_{1+}}T^{t\vartheta}f(x)\vartheta^adS(\vartheta),\end{equation}
where $dS(\theta)$ stands for the integration on the sphere and $|S^d_{1+}|_a$ is a normalization constant, cf. \eqref{FS}. In one dimension $M_{sph,a}(f)(x,r)= T_a^rf(x)$, see e.g. \cite{ss}. We also remark, that by the commutativity of Bessel translation we have
$$M_{sph,a}(f)(0,t)=\frac{1}{|S^d_{1+}|_a}\int_{S^d_{1+}}f(t\vartheta)\vartheta^adS(\vartheta),$$
i.e. $M_{sph,a}(f)(0,t)$ is just the average of $f$ on the sphere $S(0,t)_+$.

\subsection{Hankel transform and translation}

One of our main tool is the Hankel transform. The Hankel transform $\mathcal{H}_a$ of $f\in L^1_a$ is a function of $\xi$ defined by
\begin{equation}\label{H} \mathcal{H}_a(f,\xi)=\hat{f}(\xi)=\int_{\mathbb{R}^d_+}f(x)\mathbb{j}_a(x,\xi)x^adx.\end{equation}
$\mathcal{H}_a$ maps $\mathcal{S}_{e}$ to $\mathcal{S}_{e}$ and is bounded in $L^2_a$, i.e. it can be extended from the dense subspace to $L^2_a$, as in the Fourier case. It maps $L^p_a$ to the dual space $L^{p'}_a$, $1\le p\le 2$, moreover the next inversion formula holds.
\begin{proposition}\label{hank}\cite[page 38]{ss} If $f\in L^1_a(\mathbb{R}^d_+)$ is of bounded variation in a neighborhood of a point $x$ of continuity of $f$, then
\begin{equation}\label{bi}f(x)=\mathcal{H}_a^{-1}(\hat{f}(\xi))(x)=\frac{2^{d-|a|}}{\prod_{i=1}^d\Gamma^2(\alpha_i+1)}\int_{\mathbb{R}^d_+}\hat{f}(\xi)\mathbb{j}_a(x,\xi)\xi^ad\xi.\end{equation}\end{proposition}

The inverse Hankel transform of $f$ is denoted by
$$\mathcal{H}_a^{-1}f=:\check{f}.$$
 Subsequently we need the Hankel transform of a radial function. We start with the following formula. Let $S_{1+}^d$ be the positive part of the $d$-dimensional unite sphere.
 \begin{equation}\label{besatl} \frac{1}{|S_{1+}^d|_a}\int_{S_{1+}^d}\mathbb{j}_a(r\Theta,\xi))\Theta^adS(\Theta)=j_{\frac{d+|a|}{2}-1}(r|\xi|),\end{equation}
  where
 \begin{equation}\label{FS}|S_{1+}^d|_a=\int_{S_{1+}^n}\Theta^adS(\Theta)=\frac{\prod_{i=1}^d \Gamma(\alpha_i+1)}{2^{d-1}\Gamma\left(\frac{d+|a|}{2}\right)},\end{equation}
 and $dS$ stands for the integration on the sphere, cf.\cite[(3.140)]{ss}.

 Then let $f$ be a radial function, i.e. $f(x)=\varphi(|x|)$, where $\varphi$ is a function of one variable. Recalling that $r=|x|$, $\mathcal{H}_a(f)$ is also radial, furthermore we have
\begin{equation}\label{rh}\mathcal{H}_a(f)(\xi)=|S_{1+}^d|_a\int_0^\infty \varphi(r)j_{\frac{d+|a|}{2}-1}(|\xi|r)r^{d+|a|-1}dr,\end{equation}
 see \cite[Lemma 3.2]{esk}.

 \medskip

 \note
 For simplicity we introduce the following abbreviation.
 $$\gamma:=\frac{d+|a|}{2}-1.$$

 \medskip

In the sequel we use the following observation. If $f\in L^1_a$ then we have
\begin{equation}\label{szm}\mathcal{H}_a(M_{sph,a}f(0,r))(\varrho)=M_{sph,a}\hat{f}(0,\varrho).\end{equation}
Indeed, denoting by $\xi=\varrho\psi$, in view of \eqref{besatl}, \eqref{sphm} and \eqref{trid} we have
$$M_{sph,a}(\hat{f})(0,\varrho)=\frac{1}{|S^d_{1+}|_a}\int_{S^d_{1+}}\int_{\mathbb{R}^d_+}f(x)\mathbb{j}_a(x,\xi)x^adx\psi^adS(\psi)=\int_{\mathbb{R}^d_+}f(x)j_{\gamma}(r\varrho)x^adx$$ $$=|S^d_{1+}|_a\int_0^\infty \frac{1}{|S^d_{1+}|_a}\int_{S^d_{1+}}f(r\theta)\theta^adS(\theta)j_{\gamma}(r\varrho)r^{2\gamma+1}dr=(\mathcal{H}_a(M_{sph,a}(f)(0,r))(\varrho).$$

We also need the next formula.
\begin{proposition}\label{t73}\cite[Theorem 73]{ss} For all $f\in\mathcal{S}_e$ we have
\begin{equation}\label{Hsphm}M_{sph,a}(f)(x,t)=\mathcal{H}_a^{-1}\left(j_{\gamma}(t|\xi|)\mathcal{H}_a(f)(\xi)\right)(x).\end{equation}
\end{proposition}

\medskip

The relationship of Hankel transform to translation and convolution is as follows. For $f\in\mathcal{S}_e'$
\begin{equation}\label{HT}\mathcal{H}_a(T_a^yf(x))(\xi)=\mathbb{j}_a(\xi,y)\hat{f}(\xi); \ws \ws \ws T_a^y\hat{f}(\xi)=\mathcal{H}_a(\mathbb{j}_a(x,y)f(x))(\xi),\end{equation}
see e.g. \cite[(3.158)]{ss}.

Bessel convolution (cf. \eqref{conv1}) possesses all the main properties of the standard one, i.e.
\begin{equation}\label{cca}f*_a g =g*_a f,\ws \ws f*_a(g*_ah)= (f*_ag)*_ah.\end{equation}
Furthermore, Young's inequality also holds, i. e. if $1\le p, q, r \le \infty$ with $\frac{1}{r}=\frac{1}{p}+\frac{1}{q}-1$; if $f\in L^p_a$ and $g \in L^q_a$, then
\begin{equation}\label{Y}\|f*_ag\|_{r,a} \le \|f\|_{p,a}\|g\|_{q,a},\end{equation}
see \cite[(3.178)]{ss}.

The following property shows that the effect of the Hankel transform on the Bessel convolution coincides with the effect of the Fourier transform on the standard convolution, which is crucial in these studies. For $f,g \in \mathcal{S}_e$ we have
\begin{equation}\label{CH} \mathcal{H}_a(f*_ag)=\hat{f}\hat{g},\end{equation}
see e.g. \cite[(3.176)]{ss}.

Now we list the Hankel transforms of some functions which will be useful subsequently.

 Let $\alpha,\mu >-1$, $(a^2-x^2)^\mu_+=\left\{\begin{array}{ll}(a^2-x^2)^\mu, \ws 0\le x\le a,\\ 0, \ws x>a\end{array}\right.$. Then we have
\begin{equation}\label{H(Imu)}\mathcal{H}_\alpha((a^2-x^2)^{\mu}_+)(y)=\int_0^a(a^2-x^2)^\mu j_\alpha(xy)x^{2\alpha+1}dx\end{equation} $$=\frac{2^\alpha\Gamma(\alpha+1)2^\mu\Gamma(\mu+1)}{\Gamma(\alpha+\mu+2)}a^{\alpha+\mu+1}j_{\alpha+\mu+1}(ay),$$
in particular
\begin{equation}\label{H(I)}\mathcal{H}_\alpha(\chi_{[0,a]})(y)=\frac{2^\alpha a^{\alpha+1}}{\alpha+1}j_{\alpha+1}(ay),\end{equation}
see \cite[(4.38)]{o}.

\begin{equation}\label{HF} \int_0^\infty e^{-px^2}j_{\nu}(cx)x^{2\nu+1}dx=\frac{\Gamma(\nu+1)}{2p^{\nu+1}}e^{-\frac{c^2}{4p}},\end{equation}
see \cite[2.12.9 (3)]{pbm} and \eqref{B}.

\subsection{Fractional Bessel operator}

The Bessel-Laplace operator on $\mathbb{R}^{d}_{+}$ is defined as follows.
\begin{equation}\label{bdo}\Delta_a:=\sum_{i=1}^d B_{a_i}, \ws \mbox{where} \ws B_{\alpha_i}:=\frac{\partial^2}{\partial x_i^2}+\frac{2\alpha_i+1}{x_i}\frac{\partial}{\partial x_i},\end{equation}
and
\begin{equation}\label{alfa} a_i=2\alpha_i+1, \end{equation}
i.e. $\alpha_i>-\frac{1}{2}$, $i=1, \dots , d$. If $d=1$, $\Delta_a=B_\alpha$.

Let $f\in \mathcal{S}_e$. The relation of the Bessel-Laplace operator (cf. \eqref{bdo}) and of the Hankel transform is as follows.
$$\mathcal{H}_a(-\Delta_af)(\xi)=|\xi|^2\hat{f}(\xi),$$
see e.g. \cite{p}.

Let $\varphi \in \mathcal{S}_e$, $0<s<2$. Then the fractional power of the Bessel-Laplace operator is
\begin{equation}\label{fracl} (-\Delta_a)^{\frac{s}{2}}\varphi(x)=C(d,a,s)\int_{\mathbb{R}^d_{+}} \frac{\varphi(x)-T^y_a\varphi(x)}{|y|^{d+|a|+s}}y^ady, \end{equation}
see \cite{l} and in one dimension \cite{bg}. Besides other different definitions, see e.g. \cite{ch1} and the references therein, for an $f\in\mathcal{S}_e'$ we can define $(-\Delta_a)^{\frac{1}{2}}f$ as via Hankel transform as follows. Let $\varphi$ as above. Then
\begin{equation}\label{fraclh}\left(\mathcal{H}_a(-\Delta_a)^{\frac{s}{2}}\varphi\right)(\xi)=|\xi|^s\hat{\varphi}(\xi),\end{equation}
see \cite{l}.

\medskip
 We define he fractional power of the Bessel-Laplace operator with exponent greater then one as $(-\Delta_a)^{k+\frac{s}{2}}f:=(-\Delta_a)^{\frac{s}{2}}((-\Delta_a)^{k}f)$.  For functions $\varphi \in \mathcal{S}_e$, $f\in\mathcal{S}_e'$ we have
\begin{equation}\label{cs}(-\Delta_a)^{\frac{s}{2}}(\varphi*_af)=(-\Delta_a)^{\frac{s}{2}}(\varphi)*_af=\varphi*_a((-\Delta_a)^{\frac{s}{2}}f),\end{equation}
see \cite[(50)]{ch}.

The other important property of the Bessel operator is the commutativity with translation, i.e. for $f\in \mathcal{S}_e'(\mathbb{R}_+)$ $B_{\alpha}(T_\alpha^{y}f)=T_\alpha^{y}(B_{\alpha}f)$, see e.g. \cite[(2.20)]{p}. Thus, with $|y|=t$ we have
\begin{equation}\label{BT}(\Delta_a)_xM_{sph,a}(f)(x,t)=M_{sph,a}((\Delta_a)_xf))(x,t).\end{equation}
On the other hand we have the next proposition.

\begin{proposition}\label{intw0}\cite[Theorem 36]{ss} Let $f\in C^2_e$. The weighted spherical mean is the transmutation operator intertwining $\Delta_a$ and $B_\gamma$, i.e.
\begin{equation}\label{intw}(B_\gamma)_t M_{sph,a}(f)(x,t)=M_{sph,a}((\Delta_a)_xf)(x,t).\end{equation}\end{proposition}

\subsection{Spaces, continuation}
The Bessel-Sobolev spaces on $\mathbb{R}^d_{+}$ for $1\le p\le \infty$ are
$$W^{m,p}_{\Delta_a}:=\left\{f: \mathbb{R}^d_{+} \to \mathbb{R} : \Delta_a^k f \in  L^p_a, \ws k=0, \dots , m \right\},$$ $$ \|f\|_{W^{m,p}_{\Delta_a}}=\left(\sum_{k=0}^m \|\Delta_a^k f\|_{p,a}^p\right)^{\frac{1}{p}}.$$
As in the standard case, derivations are understood in the weak sense. For more information on Bessel-Sobolev spaces see \cite{ch}.

\begin{proposition}\label{gammadelta}\cite[Lemma 3.9]{ch1}
Let $s \in (0,2)$, $1\le p<\infty$. If $f\in \mathcal{S}_{e}$, then
\begin{equation}\label{D}\left\|(-\Delta_a)^{\frac{s}{2}}f\right\|_{p,a}\le c \|f\|_{W^{1,p}_{\Delta_a}},\end{equation}
where $c=c(n,a,p)$.
\end{proposition}

Below we deal with bandlimited functions, so we need the next function spaces.
Let us denote by $E_e$ the set of even entire functions. We also define the space
\begin{equation}M(\nu, p,a):=\{h\in L^p_a\cap E_e: \mathrm{supp}\hat{h}\subset B(0,\nu)\}.\end{equation}
In one dimension the following proposition is proven.

\begin{proposition}\label{ent}\cite[Theorems 2.1 and 2.2]{p} Let $\varphi \in \mathcal{S}_e(\mathbb{R_+})$, and let $f$ be in its dual space, i.e. $f \in \mathcal{S}'_e(\mathbb{R_+})$. Then we have the  following equivalences.
\begin{equation}\mathrm{supp}\varphi \in [0,\nu]\ws \Leftrightarrow \ws \hat{\varphi}\in E_e, \ws \forall \ws m\in\mathbb{N} \ws \exists \ws c_m, \ws |\hat{\varphi}(\xi)|\le\frac{c_m}{(1+|\xi|)^m}e^{\nu |\Im \xi|}, \ws \xi\in\mathbb{C}.\end{equation}
and
\begin{equation}\mathrm{supp}f \in [0,\nu]\ws \Leftrightarrow \ws \hat{f}\in E_e, \ws \exists \ws C,N, \ws |\hat{f}(\xi)|\le C(1+|\xi|)^N e^{\nu |\Im \xi|}, \ws \xi\in\mathbb{C}.\end{equation}
\end{proposition}

{\bf Remark.}
In view of Propositions \ref{ent} and \ref{hank} if $f\in L^1_a(\mathbb{R}^d_+)$ is continuous and of locally bounded variation and $\hat{f}$ is supported on the closure of the ball $B(0,\nu)$, then $M_{sph,a}f(r,0)=\mathcal{H}_a^{-1}(M_{sph,a}\hat{f}(\varrho,0))(r)$ is an entire function of exponential type $\nu$.

\subsection{Summation}

To state the main theorems, we also need to introduce some summation methods. Let $S=\sum_{k=0}^\infty$ be a formal series. The Ces\`aro means of $S$ are
\begin{equation}\label{ces}\sigma_n(S):= \frac{1}{n}\sum_{k=0}^{n-1}S_k,\end{equation}
where $S_k$ stands for the partial sums of $S$. We also define the de la Vall\'ee Poussin means $V_n$, of $S$ as follows.
\begin{equation}\label{del}V_n(S)=2\sigma_{2n(S)}-{\sigma}_n(S)=\frac{1}{n}\sum_{k=n}^{2n-1}S_k.\end{equation}
We also use the notation $\sigma_n(f)$ or $V_n(f)$ if $S$ is an expansion of $f$.

\medskip

Subsequently the following elementary observation will be useful. Let $n>m$ be integers, $p_m$ is any polynomial of degree $m$. Then
\begin{equation}\label{f-sig} f-\sigma_n(f)=f-p_m+\frac{m}{n}p_m-\frac{m+1}{m}\sigma_{m+1}(p_m)+\sigma_n(p_m-f).\end{equation}

\subsection{Chromatic derivative and expansion}

Now we define chromatic derivatives and expansions. Following the original chain of ideas, for a suitable $f$  and a polynomial $p$  we have
\begin{equation}\label{chro1}p((-\Delta_a)^{\frac{1}{2}})f(x)=\mathcal{H}_a^{-1}\left(p(|\xi|)\hat{f}(\xi)\right)(x)=c(d,a)\int_{\mathbb{R}^d_+}p(|\xi|)\hat{f}(\xi)\mathbb{j}_a(x,\xi)\xi^ad\xi.\end{equation}
Thus if $\{p_k\}$ are the orthonormal polynomial with respect to a weight $w$, the $k$th Bessel-chromatic derivative of $f$ with respect to a weight $w$ is
\begin{equation}\label{chro2}K_{B,w}^k(f)(x):=c(d,a)\int_{\mathbb{R}^d_+}p(|\xi|)\hat{f}(\xi)w(\xi)\mathbb{j}_a(x,\xi)\xi^ad\xi.\end{equation}
If $w\equiv 1$, we write $K_{B}^k(f)$.

With the Bessel-chromatic derivatives of an appropriate fixed function $\varphi$, we introduce the following sequence of functions.
$$\varphi_k:=K_{B}^k(\varphi), \ws \ws k=0,1,\dots .$$
The Bessel-chromatic expansion of $f$ with respect to $\varphi$ is
\begin{equation}\label{chro6} CE_{B,\varphi,w}(f)=\sum_{k=0}^\infty K_{B,w}^k(f)(0)K_{B}^k(\varphi).\end{equation}

\section{Chromatic expansion, weighted spherical mean at a general point}

The symmetry of the Bessel-Laplace operator ensures that on the right-hand side of \eqref{chro1} we get radial polynomials. Indeed,
\begin{equation}\label{elso}p_k((-\Delta_a)^{\frac{1}{2}})f=\int_{\mathbb{R}^d_+}\hat{f}(\xi)p_k(|\xi|)\mathbb{j}_a(x,\xi)\xi^adx.\end{equation}

This observation makes sense to take into consideration radial functions, or the spherical mean of a function.
To this we need the following modification of Proposition \ref{intw0}.

\medskip

\begin{lemma} Let $f\in W^{k+1,p}_{\Delta_a}$, $k\in\mathbb{N}$, $0<s<2$, and $\gamma$ be as above. Then we have
\begin{equation}\label{fBsphm}(-B_{\gamma})^{k+\frac{s}{2}}_t M_{sph,a}(f)(x,t)=M_{sph,a}((-\Delta_a)^{k+\frac{s}{2}}_xf)(x,t).\end{equation}
\end{lemma}

\proof
Iteration and Proposition \ref{intw0} imply that it is enough to prove for $k=0$. First let $f\in\mathcal{S}_e$.
The inversion formula for Hankel transform is obviously valid, thus it is enough to show the equality of the Hankel transform at $x$ of both sides.
In view of \eqref{Hsphm}, recalling that $|\xi|=\varrho$, we have
$$\mathcal{H}_a(M_{sph,a}((-\Delta_a)^{\frac{s}{2}}_xf)(\cdot,t))(\xi)=\varrho^s\hat{f}(\xi)j_\gamma(t\varrho).$$
On the other hand, taking into consideration that\\ $(-B_{\gamma})^{\frac{s}{2}}_tj_\gamma(t\varrho)=\varrho^s(\xi)j_\gamma(t\varrho)$, cf. \cite[(3.4)]{bg}, and \eqref{Hsphm} we have
$$\mathcal{H}_a((-B_{\gamma})^{\frac{s}{2}}_t) M_{sph,a}(f)(\cdot,t))(\xi)=(-B_{\gamma})^{\frac{s}{2}}_t\mathcal{H}_a(M_{sph,a}(f)(\cdot,t))(\xi)$$ $$=(-B_{\gamma})^{\frac{s}{2}}_tj_\gamma(t\varrho)\hat{f}(\xi)=\varrho^sj_\gamma(t\varrho)\hat{f}(\xi).$$
Then we take into consideration that $\mathcal{S}_e$ is dense in $W^{1,p}_{\Delta_a}$, see \cite[Theorem 2]{ch} and by Proposition \ref{gammadelta} the operators are bounded in $W^{1,p}_{\Delta_a}$-norm. Thus we can extend the result to $W^{1,p}_{\Delta_a}$.

\medskip

\noindent {\bf Remark.}

Taking the Hankel transform of both sides, it can be seen that similarly to \eqref{Bcser} for $f\in \mathcal{S}_e'$ translation and the fractional Bessel-Laplacian commute, i.e.
$$T^t((-\Delta_a)^{\frac{s}{2}}f)=(-\Delta_a)^{\frac{s}{2}}(T^tf).$$
Since for a regular distribution $f$ (a distribution which is defined by a locally integrable function) by interchanging the integrals we immediately get that for any $g\in \mathcal{S}_e$, $\langle M_{sph,a}(f),g\rangle_a=\langle f, M_{sph,a}(g)\rangle_a$, as usual, we define the spherical mean of an $f\in \mathcal{S}_e'$ such that for any $g\in \mathcal{S}_e$
$$\langle M_{sph,a}(f),g\rangle_a=\langle f, M_{sph,a}(g)\rangle_a.$$
Thus, with the notation above, by Parseval's formula and by \eqref{Hsphm} we have
$$\langle M_{sph,a}(f)(\cdot,t),g\rangle_a=\langle f, \mathcal{H}_a^{-1}\left(j_{\gamma}(t|\cdot|)\hat{g}(\cdot)\right)\rangle_a=\langle \hat{f}j_{\gamma}(t|\cdot|),\hat{g}\rangle_a$$ $$=\langle \mathcal{H}_a^{-1}\left(\hat{f}(\cdot)j_{\gamma}(t|\cdot|)\right),g\rangle_a.$$
Thus \eqref{Hsphm} is valid for $f\in \mathcal{S}_e'$ as well, and from the first part of the proof it can be readily seen that equation \eqref{fBsphm} is valid in distribution sense for any $f\in \mathcal{S}_e'$ too.

\medskip

Now  we are in position to define the weighted spherical Bessel chromatic expansion at a general point.

Let $w$ be a weight function, i.e. a nonnegative function with finite moments on $[0,\nu)$ with $0\le\nu\le \infty$. Let $y=t\vartheta$, $\gamma$ be as above. Denoting by $\Phi_x(t):=M_{sph,a}(f)(x,t)$, in view of\eqref{fBsphm} with $\check{w}(|y|):=\mathcal{H}_a^{-1}(w(|\eta|)$ in distribution sense we have
\begin{equation}\label{sphx}p((-B_{\gamma})^{\frac{1}{2}}_t)(\Phi_x(\cdot)*_\gamma \check{w})(t)$$ $$=c(\gamma)\int_0^\infty p(\tau)\widehat{\Phi_x(\cdot)}(\tau)w(\tau)j_{\gamma}(t\tau)\tau^{2\gamma+1}d\tau.\end{equation}
Let $\{p_k(w_1)\}_{k=0}^\infty$ be the ONP on $[0,\nu)$ with respect to $w_1(\tau)=w(\tau)\tau^{2\gamma+1}$. Then
$$p_k((-B_{\gamma})^{\frac{1}{2}}_t)(\Phi_x(\cdot)*_\gamma\check{w})(0)=c_k(\widehat{\Phi_x(\cdot)},w_1)=:c_k(x),$$
i.e. the Fourier coefficients of $\widehat{\Phi_x(\cdot)}$ with respect to $\{p_k\}=\{p_k(w_1)\}$. Introduce the notation
\begin{equation}\label{fs1} \varphi(t):=c(\gamma)\int_0^\infty w(\tau)j_{\gamma}(t\tau)\tau^{2\gamma+1}d\tau,\end{equation}
\begin{equation}\label{fs2}\varphi_k(t):=c(\gamma)\int_0^\infty p_k(\tau)w(\tau)j_{\gamma}(t\tau)\tau^{2\gamma+1}d\tau=K^k_{B}(\varphi)(t),\end{equation}
and
\begin{equation}c_k(x)=K^k_{B,w}(\Phi_x(\cdot))(0).\end{equation}
Thus, we get the following spherical Bessel-chromatic expansion.
\begin{equation}\label{chm}SCE_{B,w}(f)(t):= \sum_{k=0}^\infty c_k(x)K^k_{B}(\varphi)(t)=\sum_{k=0}^\infty K^k_{B,w}(\Phi_x(\cdot))(0)  K^k_{B}(\varphi)(t).\end{equation}
Supposing that $f$ fulfils the conditions of the inversion formula and $f\in L^2_a$,
$$SCE_{B,w}(f)(t)=(\Phi_x(\cdot)*_\gamma \check{w} )(t)$$
in $L^2_{\gamma}(\mathbb{R}_+)$ sense. If $f, \varphi \in \mathcal{S}_e$, the chromatic derivatives can be ment not only in distributional but also in ordinary sense.

Of course, it is worth taking weights which Hankel transform is known.  Below we give some examples.

\subsubsection{With Laguerre weight} At first we give the definition in distribution sense. In view of \eqref{HF}, we take $w(\tau)=e^{-\tau^2}$ on the whole semiaxis. That is we have
$$p((-B_{\gamma})^{\frac{1}{2}}_t)(\Phi_x(\cdot)*_\gamma e^{-\frac{(\cdot)^2}{4}})(t)$$ $$=\frac{1}{2\Gamma\left(\gamma+1\right)}\int_0^\infty p(\tau)\widehat{\Phi_x(\cdot)}(\tau)e^{-\tau^2}j_{\gamma}(t\tau)\tau^{2\gamma+1}d\tau,$$
 cf. \eqref{HF} and \eqref{bi}. Let $\{p_k\}_{k=0}^\infty$ be the ONP on $(0,\infty)$ with respect to the general Laguerre weight, $w_1(\tau)=e^{-\tau^2}\tau^{2\gamma+1}$. Now we have
\begin{equation}\label{fs1} \varphi(t):=\frac{2}{\Gamma(\gamma+1)}\int_0^\infty e^{-\tau^2}j_{\gamma}(t\tau)\tau^{2\gamma+1}d\tau = e^{-\frac{t^2}{4}},\end{equation}
\begin{equation}\label{fs2}\varphi_k(t):=\frac{1}{2^{\gamma-1}\Gamma(\gamma+1)\Gamma\left(\frac{\gamma+1}{2}\right)}\int_0^\infty p_k(\tau)e^{-\tau^2}j_{\gamma}(t\tau)\tau^{2\gamma+1}d\tau=K^k_B\left(e^{-\frac{t^2}{4}}\right).\end{equation}
Notice that since $e^{-\frac{t^2}{4}}$ is in $\mathcal{S}_e$, $\varphi_k$ is its $k$th Bessel-chromatic derivative in ordinary sense as well. So we have
\begin{equation}\label{chml}SCE_{B,w}(f)(t):= \sum_{k=0}^\infty K^k_B(M_{sph,a}(f)(x,\cdot))(0)  K^k_B\left(e^{-\frac{t^2}{4}}\right).\end{equation}

\medskip

After this definition of chromatic expansion we give a convergence theorem.

\medskip

\begin{theorem}\label{con3} Assume that $f\in L^1_a$, is continuous and of locally bounded variation. If $\left\|\partial_\tau \widehat{M_{sph,a}(f)_x}(\tau)\tau^{\gamma+1}e^{-\frac{\tau^2}{2}}\right\|_\infty <\infty$, then

\begin{equation}\label{c3}\left|\left(M_{sph,a}(f(x,\cdot)*_\gamma e^{-\frac{(\cdot)^2}{4}}\right)(t)-\sigma_n(SCE_{B,w}(f))(t)\right|=O(n^{-\frac{7}{32}}).\end{equation}
\end{theorem}

\medskip

For the proof we need some lemmas.

Let $I=[0,\infty)$, $\alpha>-\frac{1}{2}$, $\beta>\frac{1}{2}$; $w(x)=w(\alpha,\beta,x)=x^{\alpha}e^{-\frac{x^\beta}{2}}$, $\tilde{w}(x)=e^{-\frac{x^\beta}{2}}$.\\  The orthonormal polynomial system with respect to $w^2$ is denoted by $\{p_n(w^2, \cdot)\}_{n=0}^\infty$.
Let $a_n=a_n(w^2)$ be the Mhaskar-Rakhmanov-Saff (M-R-S) number (see e.g. \cite{lelu} and the references therein), which shows where does the norm of a weighted polynomial live. To formulate the following lemma we need a further function as follows.
$$\Psi_n(x)=\left\{\begin{array}{ll} \frac{\sqrt{x+a_nn^{-2}}(a_{2n}-x)}{n\sqrt{a_n-x+a_n\eta_n}}, \ws x\in [0,a_n]\\
\psi_n(a_n), \ws x>a_n.\end{array}\right.$$

To state the following lemma we define $n$th Christoffel function.
$$\lambda_n(w^2,x):=\inf_{p\in \Pi_n}\frac{\int_I p^2(t)w^2(t)dt}{p^2(x)}=\frac{1}{\sum_{k=0}^{n-1}p_k^2(w^2,x)}.$$

\begin{lemma}\label{lelu}\cite[Theorem 1.3 (b), Example 1]{lelu} With the above notation we have that there exists a $C>0$ such that uniformly for $n\ge 1$ and $x\in [0,\infty)$
$$\lambda_n(w^2,x)\ge c\Psi_n(x)\tilde{w}^2(x)\left(x+\frac{a_n}{n^2}\right)^{2\alpha},$$
and $\eta_n=c(\beta) n^{-\frac{2}{3}}$, $a_n=c(\beta)n^{\frac{1}{\beta}}$.
\end{lemma}

\medskip

We continue with a general lemma.

\begin{lemma}\label{gences} Let $w$ be a weight function on an real interval $I$, $fw\in L^\infty(I)\cap L^2(I)$ and $\phi \in L^\infty(I)$ arbitrary. Then
$$|\sigma_n(f,w^2)(x)w(x)\phi(x)|\le c \|fw\|_\infty \left(1+\frac{1}{n}\sup_{k=1}^n\frac{\gamma_{k-1}}{\gamma_k}\frac{w^2(x)\phi(x)}{\lambda_n(w^2,x)}\right).$$
\end{lemma}

\proof Let $x\in \left(\frac{a_n}{n^2}, \infty\right)$. By standard arguments, let us define the interval $I(n,x):=(\max\{0,x-d(n,x)\}, x+d(n,x))$, and the length is defined by the $n$th Christoffel function, $d(n,x):\frac{\lambda_n(w^2,x)}{w^2(x)}=\frac{1}{\sum_{k=0}^{n-1}p_k^2(w^2,x)w^2(x)}$. Denoting by $K_k(x,y):=\sum_{j=0}^{k-1}p_j(x)p_j(y)$, we decompose the partial sums as follows.
$$S_k(x)=\int_{I(n,x)}K_k(x,y)f(y)w^2(y)dy+\int_{\mathbb{R}_+\setminus I(n,x)}K_k(x,y)f(y)w^2(y)dy=:S_{k,1}+S_{k,2}.$$
Then, by the Cauchy-Schwartz inequality, and then taking into account the monotonicity of the Christoffel function we have
$$|S_{k,1}(x)w(x)|\le \|fw\|_\infty \int_{I(n,x)}|K_k(x,y)|w(y)dy$$ $$\le c\|fw\|_\infty \sqrt{d(n,x)}\left(\int_0^\infty\left(\sum_{j=0}^{k-1}p_j(x)p_j(y)\right)^2w^2(y)dy\right)^{\frac{1}{2}}$$ $$
\le c\|fw\|_\infty \sqrt{d(n,x)}\sqrt{\sum_{j=0}^{k-1}p_j^2(w^2,x)}w(x)\le c\|fw\|_\infty ,$$
and of course, the same estimate is valid for $\frac{1}{n}\sum_{k=0}^{n-1}|S_{k,1}(x)|w(x)$.
On the other hand, let us introduce the functions $F_{n,x}(y):=\left\{\begin{array}{ll}\frac{f(y)}{x-y}, |x-y|>d(n,x),\\ 0, |x-y|\le d(n,x).\end{array}\right.$ Then by the Christoffel-Darboux formula we have
$$\frac{1}{n}\sum_{k=0}^{n-1}|S_{k,2}(x)|=\frac{1}{n}\sum_{k=0}^{n-1}\left|\frac{\gamma_{k-1}}{\gamma_k}p_k(x)c_{k-1}(F_{n,x})-p_{k-1}(x)c_{k}(F_{n,x})\right|,$$
where $\gamma_k$ the leading coefficient of $p_k$ and $c_k(F)$ stands for the $k$th Fourier coefficient of $F$. Thus, again by the Cauchy-Schwartz inequality we have
$$\frac{1}{n}\sum_{k=0}^{n-1}|S_{k,2}(x)|w(x)\le \frac{1}{n}\sup_{k=1}^n\frac{\gamma_{k-1}}{\gamma_k}\sqrt{\sum_{k=0}^{n-1}p_k^2(w^2,x)}\sqrt{\sum_{k=0}^{n-1}c_k^2(F_{n,x})}w(x).$$
Applying Bessel's inequality the last two factors are
$$\sqrt{\sum_{k=0}^{n-1}c_k^2(F_{n,x})}w(x)\le\left(\int_{\mathbb{R}_+\setminus I(n,x)}\frac{f^2(y)}{(x-y)^2}w^2(y)dy\right)^{\frac{1}{2}}w(x)$$ $$\le c \|fw\|_\infty \frac{w(x)}{\sqrt{d(n,x)}}.$$
Summarizing, for  $x> \frac{a_n}{n^2}$ we get the following estimate.
\begin{equation}\label{sig1}|\sigma_n(f,w^2)(x)w(x)\phi(x)|\le c \|fw\|_\infty \left(1+\frac{1}{n}\sup_{k=1}^n\frac{\gamma_{k-1}}{\gamma_k}\frac{w^2(x)\phi(x)}{\lambda_n(w^2,x)}\right).\end{equation}

\medskip

Now we specify $\phi$. Let $\phi(x):=\left\{\begin{array}{ll}\sqrt{x}, 0<x<1,\\ 1, x\ge 1.\end{array}\right.$

\begin{lemma}\label{sikor} With the notation above, if $fw\in L^\infty(\mathbb{R}_+)$, then
$$\|\sigma_n(f,w^2)w\phi\|_\infty \le c \sqrt{a_n}\|fw\|_\infty.$$
\end{lemma}

\proof  According to Lemma \ref{lelu} for all $x>0$
\begin{equation}\label{lam}\frac{w^2(x)}{\lambda_n(w^2,x)}\le c\frac{x^{2\alpha}}{\left(x+\frac{a_n}{n^2}\right)^{2\alpha}}\left\{\begin{array}{ll}\frac{n \left(a_n+a_n n^{-\frac{2}{3}}-x\right)^{\frac{1}{2}}}{\sqrt{x+\frac{a_n}{n^2}}(a_{2n}-x)}, \ws x<a_n\\ \frac{n^{\frac{2}{3}}}{a_n}, \ws x\ge a_n.\end{array}\right.\end{equation}
Thus, taking into account that on infinite intervals $\sup_{k=1}^n\frac{\gamma_{k-1}}{\gamma_k}\le c a_n$ for $x> \frac{a_n}{n^2}$ we have
$$|\sigma_n(f,w^2)(x)w(x)\phi(x)|\le c \left\{\begin{array}{ll}\sqrt{a_n}h(x), \ws x< a_n,\\ n^{-\frac{1}{3}}, \ws x\ge a_n,\end{array}\right.$$
where $h(x)=\left\{\begin{array}{ll}\left(\frac{x}{x+a_n n^{-2}}\right)^{2\alpha+\frac{1}{2}}, x<1\\ \frac{x^{2\alpha}}{(x+a_n n^{-2})^{2\alpha+\frac{1}{2}}}, \ws x\ge 1, \end{array}\right.$ which is bounded.

Since a weighted polynomial attains its norm in the M-R-S interval, $I_{n,w}:=\left(\frac{a_n}{n^2}, a_n\right)$, c.f. \cite{lelu} and the references therein, in view of \eqref{sig1} and \eqref{lam}, the lemma is proven.

\medskip

To continue the proof we need the following lemma.

Let $\varphi(x):=\sqrt{x}$. The best approximation of a function $f$ in $p$-norm with polynomials of degree at most $n$ is
\begin{equation}\label{bestpol}E_n(f)_{w,p}:=\min_{p\in\Pi_n}\|(f-p)w\|_p.\end{equation}

\begin{lemma}\label{ms}\cite[Proposition 4.1]{ms} Let $w=w_{\alpha,\beta}$ as above. For each function $f\in W^p_1(w)$, $1\le p \le \infty$, we have
$$E_m(f)_{w,p}\le C\frac{\sqrt{a_m}}{m}\|f'\varphi w\|_p.$$
\end{lemma}

\medskip

\begin{lemma}\label{dbecs}  With the notation of Lemma \ref{sikor}, let $\beta=2$. If $f$ is differentiable such that $fw,f'w\varphi \in L^\infty(\mathbb{R}_+)$, $fw\in  L^2(\mathbb{R}_+)$, then
$$\|(f-\sigma_n(f,w^2)w\varphi\|_\infty=O(n^{-\frac{7}{32}}).$$
\end{lemma}

\proof
Recalling the definition of the uniformly best approximating polynomial of degree $n$ $P_{n,u,w}$ i.e. $\|(f-P_{n,u,w})w\|_\infty=E_n(f)_w$, where $E_n(f)_w$ stands for $E_n(f)_{w,\infty}$, by \eqref{f-sig} we have the estimation
$$\|(f-\sigma_n(f,w^2)w\phi\|\le \|(f-P_{m,u,w})w\phi\|+\|(P_{m,u,w}-\sigma_n(P_{m,u,w},w^2))w\phi\|$$ $$\|\sigma_n(P_{m,u,w},w^2-f)w\phi\|=\Sigma_1+\Sigma_2+\Sigma_3,$$
where $\|\cdot\|$ stands for infinity-norm. Simple calculations together with Lemma \ref{sikor} give for $m<n$,
$$\Sigma_2=\le c\left(\frac{(m+1)m^\frac{1}{4}}{n}+\frac{m}{n}\right)\|P_{m,u,w})w\phi\|.$$
Since $\phi$ is bounded, and being the best approximating polynomial, $\|P_{m,u,w})w\|$ is bounded as well, $\Sigma_2\le c \frac{m^{\frac{5}{4}}}{n}$.   $\|\Sigma_1w\phi\|\le E_m(f)w$; and by Lemma \ref{sikor} $\|\Sigma_3w\phi\|\le n^{\frac{1}{4}}E_m(f)w$. According to Lemma \ref{ms}
$$E_m(f)_w\le c\frac{\sqrt{a_m}}{m}\|f'w\varphi\|.$$
Thus, selecting $m=n^{\frac{5}{8}}$, the lemma is proven.

\medskip

\proof (of Theorem \ref{con3}.) Since $f\in L^1_a(\mathbb{R}^d_+)$, $M_{sph,a}(f)_x\in L_{2\gamma+1}^1(\mathbb{R}_+)$ and by \eqref{Y} $M_{sph,a}(f(x,\cdot)*_\gamma e^{-\frac{(\cdot)^2}{4}}\in L_{2\gamma+1}^1(\mathbb{R}_+)$ as well, and the continuity properties are also inherited. Thus the inversion formula holds and the inverse image can be expressed in the stated convolution form. We also have that $\widehat{M_{sph,a}(f)_x}\in L^\infty(\mathbb{R}_+)$, thus with $w^2(\tau)=e^{-\tau^2}\tau^{2\gamma+1}$,  $\widehat{M_{sph,a}(f)_x}\in L^2_{w^2}$,  that is  $\widehat{M_{sph,a}(f)_x}$ can be expanded.
We also have that $\widehat{M_{sph,a}(f)_x}w\in L^\infty(\mathbb{R}_+)$ which, together with the  assumptions of the theorem ensure that the conditions of Lemma \ref{dbecs} are fulfilled.\\ That is the left-hand side of \eqref{c3} is just
$$c(\gamma)\left|\int_0^\infty (\widehat{M_{sph,a}(f)_x(}\tau)-\sigma_n(\widehat{M_{sph,a}(f)_x(}\tau))j_\gamma(t\tau)e^{-\tau^2}\tau^{2\gamma+1}d\tau\right|=:D_n$$
Since $|j_\gamma(t\tau)|\le 1$, denoting by $\tilde{\phi}(\tau)=\left\{\begin{array}{ll}1,\ws \tau<1\\\tau^{-\frac{1}{2}}, \ws \tau\ge 1\end{array}\right.$,  by Lemma \ref{dbecs} we have
$$D_n\le \|(\widehat{M_{sph,a}(f)_x}-\sigma_n(\widehat{M_{sph,a}(f)_x})w\phi\|_\infty \int_0^\infty \tau^\gamma\tilde{\phi}(\tau)e^{-\frac{\tau^2}{2}}d\tau\le c n^{-\frac{7}{32}}.$$

\noindent{\bf Remark.}

\noindent (1) To estimate the speed of convergence it is enough to assume that  $M_{sph,a}(f(x,\cdot)*_\gamma e^{-\frac{(\cdot)^2}{4}}\in L_{2\gamma+1}^1(\mathbb{R}_+)$. The condition $f\in L^1_a$ is necessary for the convolution form of the left-hand side.

\noindent (2) Condition (*): $\left\|\partial_\tau \widehat{M_{sph,a}(f)_x}(\tau)\tau^{\gamma+1}e^{-\frac{\tau^2}{2}}\right\|_\infty <\infty$ can be replaced by\\
(**): $M_{sph,a}(f)_x\in L^1_{2\gamma+3}(\mathbb{R}_+)$. Indeed, in view of \eqref{bder}
$$\partial_\tau \widehat{M_{sph,a}(f)_x}(\tau)=\frac{-\tau}{2(\gamma+1)}\int_0^\infty M_{sph,a}(f)_x(t)j_{\gamma+1}(t\tau)t^{2(\gamma+1)+1}dt.$$
Thus  by (**), $(*)\le c\|\tau^{\gamma+2}e^{-\frac{\tau^2}{2}}\|_\infty <\infty$.

\noindent (3) The estimation of the Christoffel function (cf. Lemma \ref{lelu}) is given for more general weights of form $x^{2\varrho}e^{-Q(x)}$ on $[0,\nu)$, where $\nu$ if finite or infinite. Thus with an appropriate $\phi$ one can give estimations similar to Lemma \ref{sikor} which implies theorems similar to Theorem \ref{con3}. Our choice is motivated by the explicit form of the inverse Hankel transform of $w$.

\medskip

\subsubsection{Bandlimited case}
In this subsection we investigate functions under the assumption $\mathrm{supp}\widehat{\Phi_x(\cdot)}\subset [0,\nu]$.\\
Assume that $\widehat{\Phi_x}\in L^2_{2\gamma+1}([0,\nu])$. Let our orthonormal system on $[0,\nu]$ be the system of transformed Jacobi polynomials, i.e.
\begin{equation}\label{modjac}p_k(\varrho)=cp_k^{(0,\gamma)}\left(\frac{2\varrho}{\nu}-1\right),\end{equation}
where the original Jacobi polynomials, $\left\{p_k^{(\alpha,\beta)}\right\}_{k=0}^\infty$ are the orthonormal polynomials on $[-1,1]$ with respect to the weight \begin{equation}\label{jacw}w^{(\alpha,\beta)}(x)=(1-x)^\alpha(1+x)^\beta,\ws \ws \alpha,\beta>-1.\end{equation}
Then, with the above notation we have
$$p_k((-B_{\gamma})^{\frac{1}{2}}_t)(\Phi_x(t)=\int_0^\nu p_k(\tau)\widehat{\Phi_x}(\tau)j_\gamma(t\tau)\tau^{2\gamma+1}d\tau.$$
Since $\widehat{\Phi_x}\in L^2_\gamma([0,\nu])$,
$$\widehat{\Phi_x}(\tau)=\sum_{k=0}^\infty c_k(x)p_k(\tau),$$
in $L^2_{2\gamma+1}$-sense, where
$$c_k(x)=p_k((-B_{\gamma})^{\frac{1}{2}}_t)(\Phi_x(0)=\int_0^\nu p_k(\tau)\widehat{\Phi_x}(\tau)\tau^{2\gamma+1}d\tau=:K^k_{B}(\Phi_x(\cdot))(0).$$
As above, we define $\varphi$ and $\varphi_k$.
$$\varphi(t):=\mathcal{H}^{-1}_\gamma(\chi_{[0,\nu]})=c(\gamma)\int_0^\nu j_{\gamma}(t\tau)\tau^{2\gamma+1}d\tau=c(\nu,d,a)j_{\frac{d+|a|}{2}}(\nu t),$$
where $c(\nu,d,a)=\frac{\nu^{\frac{d+|a|}{2}}}{2^{d+|a|-3}(d+|a|)\Gamma^2\left(\frac{d+|a|}{2}\right)}$, cf. \eqref{H(I)}.
$$\varphi_k(t):=\mathcal{H}^{-1}_\gamma(p_k\chi_{[0,\nu]})= c(\gamma)\int_0^\nu p_k(\tau)j_{\gamma}(t\tau)\tau^{2\gamma+1}d\tau=K^k_{B}(\varphi)(t).$$
Now, the chromatic expansion of the spherical mean of $f$ is
\begin{equation}\label{chrosph} M_{sph,a}(f)(x,t)=c(\nu,d,a)\sum_{k=0}^\infty K^k_{B}(M_{sph,a}(f)_x(\cdot))(0)K^k_{B}(j_{\gamma+1}(\nu \cdot))(t).\end{equation}
Since  $\widehat{\Phi_x}$ is expanded in a finite interval, its Jacobi-Fourier series is convergent in $L^1_{2\gamma+1}$-sense as well, thus application of the inverse Hankel transform implies that the convergence in \eqref{chrosph} is uniform.

Now, to recover the function from its spherical chromatic expansion we need the inversion formula for the weighted spherical mean.

\medskip

\begin{lemma}\label{spheinv}\cite[Theorem 5]{s} Let $m=\left[\frac{d+|a|-1}{2}\right]+1$. Let $f\in C^{2m}(\mathbb{R}^d_+)$, such that $\frac{\partial}{\partial x_i}f|_{x_i=0}=0$, $i=1, \dots , d$.  Then
$$h(u)f(x)=\frac{\sqrt{\pi}}{2^{2(m-1)}\Gamma(m)\Gamma\left(\frac{2m-d-|a|+1}{2}\right)\Gamma\left(\frac{d+|a|}{2}\right)}$$ $$\times \left(\frac{\partial^2}{\partial u^2}-\Delta_{a,x}\right)^m\int_0^\infty h(u-s)\int_0^s (s^2-t^2)^{\frac{2m-d-|a|-1}{2}}M_{sph,a}(f)(x,t)t^{2\gamma+1}dtds,$$
where $h\in C^{2m}(\mathbb{R})$ is an arbitrary function such that the integral on the right-hand side is convergent.
\end{lemma}

\medskip

Together with the uniform convergence of series in \eqref{chrosph} Lemma \ref{spheinv} immediately implies the next formula.

\begin{theorem}\label{tcccc} Let $f\in W^{m,1}_{\Delta_a}\cap C^{2m}(\mathbb{R}^d_+)$ with $\frac{\partial}{\partial x_i}f|_{x_i=0}=0$, $i=1, \dots , d$. Assume that $\mathrm{supp}\widehat{M_{sph,a}(f)(x,\cdot)}\subset [0,\nu]$. Let $h\in C^{2m}(\mathbb{R})$ an appropriate function as above. Then with the notation above we have
\begin{equation}\label{ron}h(u)f(x)=c(\nu,m,\gamma)\sum_{l=0}^m\binom{m}{l}(-1)^l\sum_{k=0}^\infty c_k^{(l)}(x)b_{k,m-l}(u),\end{equation}
where
$$c(\nu,m,\gamma)=\frac{\sqrt{\pi}\nu^{\gamma+1}\left(m-\gamma-\frac{1}{2}\right)}{\Gamma(m)\Gamma\left(m+\frac{1}{2}\right)\Gamma^4(\gamma+1)(\gamma+1)2^{m+4\gamma-\frac{1}{2}}},$$
$$ c_k^{(l)}(x)=K^k_B(M_{sph,a}(\Delta_{a,x} f)(x,\cdot))(0),$$
$b_{k,m-l}(u)$ are the Jacobi-Fourier coefficient on $[0,\nu]$ of
$$g_{u,m-l}(\tau):=\int_0^\infty h^{(2(m-l))}(u-s)s^{m-\frac{1}{2}}j_{m-\frac{1}{2}}(s\tau)ds.$$
\end{theorem}

\proof Since $\sum_{k=0}^\infty c_k(x)\varphi_k(t)$ is uniformly convergent, with $\mu:=\frac{2m-d-|a|-1}{2}$ we have

$$I:=\int_0^\infty h(u-s)\int_0^s (s^2-t^2)^{\mu}M_{sph,a}(f)(x,t)t^{2\gamma+1}dtds$$ $$=c(\nu,d,a)\int_0^\infty h(u-s)\sum_{k=0}^\infty c_k(x)\int_0^s (s^2-t^2)^\mu \varphi_k(t)t^{2\gamma+1}dtds.$$
Computing the inner integral we get
$$\int_0^s (s^2-t^2)^\mu \varphi_k(t)t^{2\gamma+1}dt=\int_0^s (s^2-t^2)^\mu\int_0^\nu p_k(\tau)j_\gamma(t\tau)\tau^{2\gamma+1}d\tau t^{2\gamma+1}dt$$ $$=c(\gamma,\mu)s^{\gamma+\mu+1}\int_0^\nu p_k(\tau)j_{\gamma+\mu+1}(s\tau)\tau^{2\gamma+1}d\tau,$$
cf. \eqref{H(Imu)}. Thus
$$I=c(\nu,d,a\gamma,\mu)\int_0^\infty h(u-s)s^{\gamma+\mu+1}\sum_{k=0}^\infty c_k(x)d_k(s)ds,$$
where $d_k(s)$ are the Jacobi-Fourier coefficients of $j_{\gamma+\mu+1}(s\tau)$ with respect to $\{p_k\}$. Since $\|j_{\gamma+\mu+1}(s\cdot)\|_{2,w,[0,\nu]}\le c\nu^{2(\gamma+1)}$, independently of $s$, $\|d_k(s)\|_2\le K$ independently of $s$. Since $\widehat{M_{sph,a}(f)_x}\in L^\infty([0,\nu])$, it is also in  $L^2_w([0,\nu])$, thus $\|c_k(x)\|_2\le K(x)$. So by the Cauchy-Schwartz inequality
$$I=c(\nu,d,a\gamma,\mu)\sum_{k=0}^\infty c_k(x)\int_0^\infty h(u-s)s^{\gamma+\mu+1}\int_0^\nu p_k(\tau)j_{\gamma+\mu+1}(s\tau)\tau^{2\gamma+1}d\tau ds.$$
We can choose $h$ such that the series is uniformly convergent. To apply the differential operator to the series first we observe
$$\left(\frac{\partial^2}{\partial u^2}-\Delta_{a,x}\right)^m (c_k(x)h(u-s))=\sum_{l=0}^m\binom{m}{l}(-1)^l\Delta_a^lc_k(x)h^{(2(m-l)}(u-s),$$
where $h'(u-s)=\frac{\partial}{\partial u}$. Recalling the definition of the weighted spherical mean, since $f\in W^{m,1}_{\Delta_a}$, we have
$$\Delta_a^l c_k(x)=c_k(\widehat{M_{sph,a}(\Delta_a^l f)_x})=:c_k^{(l)}(x),$$
i.e. the $k$th Jacobi-Fourier coefficient of $\widehat{M_{sph,a}(\Delta_a^l f)_x}$ with respect to $\{p_k\}$. By the assumptions $\{c_k^{(l)}(x)\} \in l_2$ as well. Let
$$b_{k,m-l}(u):=\int_0^\nu p_k(\tau)\int_0^\infty h^{(2(m-l))}(u-s)s^{\gamma+\mu+1}j_{\gamma+\mu+1}(s\tau)ds\tau^{2\gamma+1}d\tau,$$
the $k$th Fourier coefficient of $g_{u,m-l}(\tau):=\int_0^\infty h^{(2(m-l))}(u-s)s^{\gamma+\mu+1}j_{\gamma+\mu+1}(s\tau)ds$. With an appropriate $h$, $g_{u,m-l}(\tau)\in L^2_{2\gamma+1}([0,\nu])$. Indeed, by the Minkowski inequality we have
$$\|g_{u,l}(\tau)\|_{2,w,[0,\nu]}\le \int_0^\infty |h^{(2(m-l))}(u-s)|s^{\gamma+\mu+1}\sqrt{\int_0^\nu j_{\gamma+\mu+1}(s\tau)^2\tau^{2\gamma+1}d\tau} ds$$ $$\le c\nu^{\gamma+1}\int_0^\infty |h^{(2(m-l))}(u-s)|s^{\gamma+\mu+1}ds.$$
If the last integral is convergent, for all $0\le l \le m$ we have that $\|c_k^{(l)}(x)\|_2\le C(x)$ and $\|b_{k,m-l}(u)\|_2\le K(u)$, and by the assumptions $C(x)$ and $K(u)$ are continuous. Thus
$$\sum_{k=0}^\infty c_k^{(l)}(x)b_{k,m-l}(u)$$
are locally uniformly convergent for all $0\le l \le m$. Finally,
$$h(u)f(x)=\left(\frac{\partial^2}{\partial u^2}-\Delta_{a,x}\right)^m c I=c\sum_{l=0}^m\binom{m}{l}(-1)^l\sum_{k=0}^\infty c_k^{(l)}(x)b_{k,m-l}(u),$$
which agrees with \eqref{ron}.

\medskip

\noindent {\bf Example.}

Let $h(u)=e^u$. Then $h^{(2(m-l))}(u-s)=e^u e^{-s}$ Thus
$$g_{u,m-l}(\tau)=e^u\int_0^\infty e^{-s}s^{\gamma+\mu+1}j_{m-\frac{1}{2}}(s\tau)ds$$ $$=e^u\frac{2^{m-\frac{1}{2}}\Gamma\left(m+\frac{1}{2}\right)}{\tau^m}\int_0^\infty \frac{e^{-s}}{\sqrt{s}}J_{m-\frac{1}{2}}(s\tau)\sqrt{s\tau}ds$$ $$=e^u2^{m-\frac{1}{2}}\Gamma\left(m+\frac{1}{2}\right)\frac{1}{\sqrt{1+\tau}(1+\sqrt{1+\tau^2})^{m-\frac{1}{2}}}=e^ug_m(\tau),$$
see \cite[(5.1)]{o}. Thus, with the notation of Theorem \ref{tcccc} we have
$$f(x)=c(\nu,m,\gamma)\sum_{k=0}^\infty b_{k}\sum_{l=0}^m\binom{m}{l}(-1)^l c_k^{(l)}(x),$$
where $b_{k}$ are the Jacobi-Fourier coefficients of $g_m$.

\section{Chromatic expansion, weighted spherical mean at zero}

Below we give some simplifications of the above discussed method. Let us suppose that $\hat{f}$ is supported on $\overline{B(0,\nu)}_+$ and let $w$ be a weight function on $[0,\nu]$, i.e. $w\ge 0$ and $w$ has finite moments. Then  $\{p_k\}_{k=0}^\infty$, is the system of orthonormal  polynomials on $[0,\nu]$ with respect to the weight $w_1(\varrho):=w(\varrho)\varrho^{n+|a|-1}$. Recalling $\check{w}:=\mathcal{H}_a^{-1}(w(|\xi|))$, in view of \eqref{CH} and \eqref{cs}, \eqref{chro2} is modified as
$$K_{B,w}^k(f)(x):=K_{B}^k(f)*_a\check{w}(x)$$ $$=
p_k((-\Delta_a)^{\frac{1}{2}})f*_a\check{w}(x)=\int_{B(0,\nu)_+}\hat{f}(\xi)p_k(|\xi|)(\xi)w(|\xi|)\mathbb{j}_a(x,\xi)x^adx.$$
Recall the notation $|\xi|=\varrho\psi$. Continuing the chain of ideas started at \eqref{elso},  by \eqref{trid}, in shperical coordinates, at zero it becomes
$$K_{B,w}^k(f)(0)$$ $$=|S^d_{1+}|_a\int_0^\nu p_k(\varrho)w(\varrho)\frac{1}{|S^d_{1+}|_a}\int_{S^d_{1+}}\hat{f}(\varrho\psi)\psi^adS(\psi)\varrho^{n+|a|-1}d\varrho$$ \begin{equation}\label{chro3}=|S^d_{1+}|_a\int_0^\nu p_k(\varrho)M_{sph,a}(\hat{f})(0,\varrho)w(\varrho)\varrho^{n+|a|-1}d\varrho.\end{equation}
 Thus, \eqref{chro3} implies that
\begin{equation}\label{chro4}\frac{1}{|S^d_{1+}|_a}p_k((-\Delta_a)^{\frac{1}{2}})f*_a\check{w}(0)=c_k(M_{sph,a}(\hat{f})(0,\cdot)),\end{equation}
i.e. the $k$-th coefficient in the expansion of $M_{sph,a}(\hat{f})(0,\cdot)$ with respect to the orthonormal system  $\{p_k\}$.
Let us define
\begin{equation}\label{chro5} \varphi(r):=\mathcal{H}_a^{-1}\left(w(|\xi|)\chi_{B(0,\nu)_+}(\xi)\right)(r);\end{equation}
$$\varphi_k(r):=K_{B}^k(\varphi)(r):=\mathcal{H}_a^{-1}\left(p_k(|\xi|)w(|\xi|)\chi_{B(0,\nu)_+}(\xi)\right)(r), \ws \ws k=0,1,\dots .$$
We also have
$$\varphi_k(r)=\tilde{\varphi}_k*_a\check{w}(r); \ws \ws \tilde{\varphi}_k:=\mathcal{H}_a^{-1}(p_k).$$
In view of \eqref{rh} these are radial functions. Thus we can define the weighted Bessel-chromatic expansion of $M_{sph,a}(f)(0,r)$ (i.e. the weighted spherical Bessel-chromatic expansion of $f$). Notice, that $M_{sph,a}(f)(0,r)$ is just equal to the effect of the weighted spherical mean of the weighted Dirac-delta distribution at $f$. In general, the chromatic derivatives are understood in distributional sense, as above. If $f$ fulfils the conditions of the inversion theorem, we have
\begin{equation}\label{chro60}M_{sph,a}(f)(0,\cdot)*_a\check{w}= S^0CE_{B,w}(f)(r):=\frac{1}{|S^d_{1+}|_a}\int_{S^d_{1+}}f(r\theta)\theta^adS(\theta)*_a\check{w}\end{equation} $$=(\mathcal{H}_a^{-1}(M_{sph,a}(\hat{f})(0,\varrho)w(\varrho))(r)$$
$$=\frac{1}{|S^d_{1+}|_a}\sum_{k=0}^\infty \left(p_k((-\Delta_a)^{\frac{1}{2}})(f)*_a\check{w}(0)\right)\left(\tilde{\varphi}_k(r)*_a\check{w}\right)$$ $$=\frac{1}{|S^d_{1+}|_a}\sum_{k=0}^\infty K_{B,w}^k(f)(0)K_{B}^k(\varphi)(r).$$
Since the Hankel transform maps $L^2_a$ to $L^2_a$,  in general we have at least that
$$\lim_{n\to 0}\int_0^\infty \left(\left(M_{sph,a}(f)(0,r)-\sum_{k=0}^n c_k\tilde{\varphi}_k\right)*_a\mathcal{H}_{2\gamma+1}^{-1}(\sqrt{w})\right)^2 r^{2\gamma+1}dr=0.$$
In our bandlimited case, as above, Cauchy-Swartz inequality implies $L^1_{w_1}$-convergence, and the inverse Hankel transform ensures the uniform convergence in \eqref{chro60}.
\medskip

Although the previous chain of ideas ensure uniform convergence, it does not give any information about the speed of convergence.

\medskip

\subsubsection{Without weight}
De la Vall\'ee Poussin means (cf.\eqref{del}) and the best approximating polynomials (cf. \eqref{bestpol} have proven to be effective tools for estimating the rate of convergence. Indeed, de la Vall\'ee Poussin mean, $V_n$  preserve polynomials of degree at most $n$, consequently, the following lemma holds.

\medskip

Recall the definition of the Jacobi weight function on $[-1,1]$, cf \eqref{jacw}. $V_n^{(\alpha,\beta)}$ is a de la Vall\'ee Poussin mean of the Jacobi-Fourier expansion with respect to $w^{(\alpha,\beta)}$. $\phi(x)=\sqrt{1-x^2}$ and we define the Jacobi-Sobolev space  $W^r_p(w^{(\alpha,\beta)})$.
\begin{equation}\label{W} W^r_p(w^{(\alpha,\beta)}):=\{f\in C^{r-1}[-1,1] : f^{(r)}\in L^p_{w^{(\alpha,\beta)}}\}, \ws \ws 1\le p<\infty,\end{equation}
and $ W^r_\infty(w^{(\alpha,\beta)})=C^{r-1}[-1,1]$, see \cite{mt}.

\begin{lemma}\label{ve}\cite[Theorem 2.1]{x}, \cite[Theorem 3.6]{mt} Let $1\le p<\infty$,  or $f\in C[-1,1]$ if $p=\infty$.  Then
$$\|f-V_n^{(\alpha,\beta)}\|_{p,w^{(\alpha,\beta)}}\le cE_n(f)_{p,w^{(\alpha,\beta)}},$$
and if $f\in W^s_p(w^{(\alpha,\beta)})$ for $1\le p<\infty$,  or $f\in C^s[-1,1]$ if $p=\infty$, we have
$$E_n(f)_{p,w^{(\alpha,\beta)}}\le \frac{c}{n^s}\|\phi^rf^{(s)}\|_{p,w^{(\alpha,\beta)}}.$$
\end{lemma}

\medskip

Let $\tilde{w}(\varrho)=\tilde{w}^{(0,n+|a|-1)}(\varrho)=\varrho^{n+|a|-1}$ the Jacobi weight transformed to the interval  $[0,\nu]$ and $\tilde{\phi}(\varrho)=\sqrt{\varrho(\nu-\varrho)}$ and $\{\tilde{p}_k\}$ is the modified orthonormal Jacobi system, cf. \eqref{modjac}.\\
Thus, by Lemma \ref{ve} $\| M_{sph,a}(\hat{f})^{(s)}-\tilde{V}_n^{(0,n+|a|-1)}\|_{1,\tilde{w}}\le c  \frac{c}{n^{s}}\|\tilde{\phi}^s M_{sph,a}(\hat{f})^{(s)}(0,\cdot)\|_{1,\tilde{w}}$. Thus, if the right-hand side is bounded, by the inverse Hankel transform we get the next result.

\begin{theorem}\label{t3} Let  $f\in L^1_a(\mathbb{R}^d_+)$ is of locally bounded variation and continuous, furthermore let us assume that $\hat{f}$ is supported on $B(0,\nu)$. If  $\tilde{\phi}^s M_{sph,a}(\hat{f})^{(s)}(0,\cdot)\in L^1_{\tilde{w}}([0,\nu])$, then
$$\|M_{sph,a}(f)(0,\cdot)-V_n(S^0SE_B(f))\|_\infty\le \frac{c}{n^{s}}\|\tilde{\phi}^s M_{sph,a}(\hat{f})^{(s)}(0,\cdot)\|_{1,\tilde{w}}.$$
\end{theorem}

\medskip

\subsubsection{With an extra weight} Below we give estimate the rate of convergence of the Ces\`aro means. To this we introduce the generalized Jacobi weights, see e.g. \cite{mave}.

\note
Recall $\phi(x):= \sqrt{1-x^2}$. Let $m\in\mathbb{N}$. $w_J\in GJ$ is a generalized Jacobi weight on $[-1,1]$ if
$$w_J(x)=H(x)\phi^{-1}(x)w_0(\sqrt{1-x})w_{m+1}(\sqrt{1+x})\prod_{r=1}^m w_r(|x-t_r|),$$
where $t_r\in(-1,1)\ws r=1,\dots,m$; $w_r(\delta)=\prod_{s=1}^{l_r}(\omega_{r,s}(\delta))^{\gamma_{r,s}}$, where $\gamma_{r,s}$ are positive real numbers, $\omega_{r,s}$ are concave moduli of continuity, and $H>0$ such that $H, \frac{1}{H}\in L_\infty$. Furthermore we assume that around zero $\int_0^\delta w_r(\tau)d\tau=O(\delta w_r(\delta))$, $r=0,\dots , m+1$ and the usual modulus of continuity of $ H(\cos\theta)$ satisfies\\ $\omega(H(\cos\theta),\delta)_\infty \delta^{-1}\in L^1_{[0,1]}$.

\medskip

For the corresponding Christoffel function we have the following estimation.

\begin{lemma}\label{l88}\cite[Theorem 3.1, Remark 1]{mave}, \cite[Theorem A]{mato} Let $w\in GJ$. Then
\begin{equation}\label{mavec}\lambda_n(w^2,x)\sim \frac{1}{n}w_0\left(\sqrt{1-x}+\frac{1}{n}\right)w_{m+1}\left(\sqrt{1-x}+\frac{1}{n}\right)\prod_{r=1}^m  w_r\left(|x-t_r|+\frac{1}{n}\right)\end{equation}
uniformly in $x$ and $n$.\\
If $w\in GJ$ has no inner zeros, then
\begin{equation}\label{matoe} E_n(f)_w\le \frac{c}{n^s}(\|wf^{(s)}\phi^s\|_\infty +\|wf\|_\infty).\end{equation}
\end{lemma}

Let $\mu>-1$,  $w(\varrho)=(\nu^2-\varrho^2)^\mu\chi_{[0,\nu]}(\varrho)$.  Observe, that with $(\nu+\varrho)^\mu=H(\varrho)$, $w_J:=w(\varrho)\varrho^{n+|a|-1}\in GJ$. Let $\{p_k\}_{k=0}^\infty$ be the system of orthonormal polynomials with respect to $w_J$ on $[0,\nu]$. If $M_{sph,a}(\hat{f})(0,\cdot)\in L^2_{w_J}([0,\nu])$ then it can be expanded with respect to $\{p_k\}$, i. e. there exists
$$c_k:=\int_0^\nu M_{sph,a}(\hat{f})(0,\varrho)p_k(\varrho)w_J(\varrho)d\varrho = K_{B,w}^kf(0).$$
Recalling the notation $\gamma=\frac{d+|a|}{2}-1$, in view of \eqref{H(Imu)}
$$\varphi(r)=\frac{2^{\gamma+\mu}\nu^{\gamma+\mu+1}\Gamma(\gamma+1)\Gamma(\mu+1)}{\Gamma(\gamma+\mu+1)}j_{\gamma+\mu+1}(\nu r).$$
Thus
$$SCE_{B,w}(f)(r)=c(d,a,\mu)\sum_{k=0}^\infty K_{B,w}^k(f)(0)K_{B}^k(j_{\gamma+\mu+1}(\nu r)),$$
where we also have
$$K_{B}^k(j_{\gamma+\mu+1}(\nu r))=c(n,a,\mu)\left(\mathcal{H}_\gamma^{-1}(p_k)*_\gamma j_{\gamma+\mu+1}(\nu \cdot)\right)(r).$$
With the notation $S=\sum_{k=0}^\infty c_kp_k$ and $\tilde{S}:= \sum_{k=0}^\infty c_k\tilde{\varphi}_k$, assuming that the inversion formula holds, as above, we have
$$\left|\left(\left(M_{sph,a}f(0,\cdot)-\sigma_n(\tilde{S})\right)*_\gamma j_{\gamma+\mu+1}(\nu \cdot)\right)(r)\right|$$ $$\le \int_0^\nu \left| \left(M_{sph,a}(\hat{f})(0,\varrho)-\sigma_n(S)(\varrho)\right)w(\varrho)j_\gamma(r\varrho)\varrho^{2\gamma+1}\right|d\varrho $$ $$\le c \|((M_{sph,a}(\hat{f})(0,\cdot)-\sigma_n(S))w\|_\infty.$$
Now we can formulate our theorem as follows.

\medskip

\begin{theorem}\label{t44} Let $f\in L^1_a$ be continuous and of locally bounded variation. Assume further that $\mathrm{supp}\hat{f}\in B(0,\nu)$ and $M_{sph,a}(\hat{f})^{(s)}(0,\cdot)w\phi^s\in L^\infty([0,\nu]$. Then
$$\left\|(M_{sph,a}(f)(0,\cdot)-\sigma_n(\tilde{S}))*_aj_{\gamma+\mu+1}(\nu\cdot)\right\|_\infty=O\left(n^{-\frac{s}{s+1}}\right).$$
\end{theorem}

\proof The assumptions ensure that the inverse Hankel transform reconstructs $f$, and so does for $(M_{sph,a}(f)(0,\cdot)$. Thus it is enough to estimate $(M_{sph,a}(\hat{f})(0,\cdot)-\sigma_n(S))w$.
Since on finite intervals $\sup_{k=1}^n\frac{\gamma_{k-1}}{\gamma_k}$ is bounded, Lemma \ref{gences} implies
$$|\sigma_n(f,w)(x)\sqrt{w}(x)|\le c \|f\sqrt{w}\|_\infty \left(1+\frac{1}{n}\frac{w(x)}{\lambda_n(w,x)}\right).$$
Thus in view of \eqref{mavec}
$$|\sigma_n(f,w)(x)\sqrt{w}(x)|\le c \|f\sqrt{w}\|_\infty.$$
Let $P_{m,u,w}$ be the uniform best approximating polynomial of degree $m$. Taking into account \eqref{f-sig}, and then \eqref{matoe}, we have have
$$\|((M_{sph,a}(\hat{f})(0,\cdot)-\sigma_n(S))w\| \le \|(M_{sph,a}(\hat{f})(0,\cdot)-P_{m,u,w})w\|+\frac{m}{n}\|P_{m,u,w} w\|$$ $$+\frac{m+1}{n}\|\sigma_{m+1}(P_{m,u,w})w\|+\|\sigma_n(P_{m,u,w}-(M_{sph,a}(\hat{f})(0,\cdot))w\|$$ $$\le
c\|M_{sph,a}(\hat{f})(0,\cdot)w\|\left(E_m(f)_w+\frac{m}{n}\right)$$ $$\le c(\|M_{sph,a}(\hat{f})(0,\cdot)w\|+\|M_{sph,a}(\hat{f})^{(s)}(0,\cdot)w\phi^s\|)\left(\frac{1}{m^s}+\frac{m}{n}\right),$$
where $\|\cdot\|$ stand for the infinity norm. Thus selecting $m=n^{\frac{1}{1+s}}$, the theorem is proven.

\medskip

\noindent{\bf Remark.}

\noindent (1) Since Lemma \ref{l88} refers to any general Jacobi weight, Theorem \ref{t44} can be proven in more general circumstances. As above, our choice is motivated by the explicit form of $\mathcal{H}^{-1}_a(w)$.

\noindent (2) Due to the reconstructing property of de la Vall\ 'ee Poussin means, the speed of convergence is generally better in $V_n$ case than in $\sigma_n$ case. On the other hand, the degree of $V_n$ is twice bigger than the degree of $\sigma_n$.

\medskip

\section{Compactification of the support}

As we have seen, bandwidth limitation is a fairly powerful feature with several comfortable consequences. As in the standard setup, in Bessel case we can also approximate functions with even entire function of exponential type $\nu$, say. In this subsection we summarize these results in brief.

As in \cite{ch} (see also \cite{p}), define the radial function $\eta \in \mathcal{S}_e$ such that $|\eta|\le 1$,
$$\eta(x)=\eta(|x|)=\left\{\begin{array}{ll}1, \ws |x|\le 1,\\0, \ws |x|\ge 2\end{array}\right. .$$
For an $f\in L_a^p$ ($1\le p \le \infty$), let
\begin{equation}\label{Pnu}P_\nu(f)(x):=\mathcal{H}_a^{-1}\left(\eta\left(\frac{|\xi|}{\nu}\right)\hat{f}(\xi)\right)=\varrho_\nu*_af(x),\end{equation}
where $\varrho_\nu(x):=\mathcal{H}_a^{-1}\left(\eta\left(\frac{|\xi|}{\nu}\right)\right)$.

Since  $\eta \in \mathcal{S}_e$,  $\hat{\eta} \in \mathcal{S}_e\subset L^1_a$. In view of \eqref{Y} $\|P_\nu(f)\|_{p,a}\le \|\varrho_\nu\|_{1,a}\|f\|_{p,a}$ and by a simple replacement, it can be readily seen, that $\|\varrho_\nu\|_{1,a}=\|\varrho_1\|_{1,a}$. Thus we have (cf. \cite{p})\\
$\|P_\nu(f)\|_{p,a}\le c \|f\|_{p,a}$; $c\neq c(\nu)$. We also have $\mathrm{supp}\hat{P_\nu}(f)\subset B(0,\nu)$. \\
Since $\eta \in \mathcal{S}_e$, $P_\nu(f)\in M(\nu, p,a)$. Furthermore for all $h \in M(\nu, p,a)$, $P_\nu(h)=h$\\
Now we define the best approximation of an $f\in L_a^p$ with functions from $M(\nu, p,a)$.

\medskip

\begin{defi} Let $1\le p \le \infty$.
$$E_\nu(f)_{p,a}:=\inf\{\|f-h\|_{p,a} : h \in M(\nu, p,a)\}.$$
\end{defi}

\medskip

$P_\nu(f)$ is a near-best approximation to $f$, i.e. $\|f-P_\nu(f)\|_{p,a}\le c E_\nu(f)_{p,a}$.\\
Now we define the spherical Bessel difference of a function as follows.
$$\Delta_{y,a}f(x):=T^yf(x)-f(x), \ws \ws \Delta_{sph,r,a}f(x):=\frac{1}{|S^d_{1+}|_a}\int_{S^d_{1+}}\Delta_{r\Theta,a}f(x)\Theta^adS(\Theta);$$
 Let $1\le p \le \infty$. The spherical Bessel p-modulus of smoothness is
 $$\omega_{sph,m}(f,t)_{p,a}:=\sup_{0<h<t}\|\Delta_{sph,h,a}^m f\|_{p,a}.$$

 \medskip

 \begin{lemma}\label{leef}\cite[Corollary 1, Lemma 5]{ch} Let  $1\le p \le \infty$, $\nu>0$, $m\in\mathbb{N}$. If $f\in L^p_a$, we have
 $$ E_\nu(f)_{p,a}\le c \omega_{sph,m}\left(f, \frac{1}{\nu}\right)_{p,a}.$$
 If $f\in W^{m,p}_{\Delta_a}$. Then we have
 $$E_\nu(f)_{p,a}\le \frac{c}{\nu^{2m}}\|\Delta_a^m f\|_{p,a},$$
where $c=c(n,a,m)$.
 \end{lemma}

  (in dimension 1 see also \cite{p})

After this preparation we are in position to define the $\nu$-chromatic expansion of a function which is not band limited, i.e. the support of $\hat{f}$ is not compact, as follows.

\begin{equation}\label{chro7}S^{(0)}CE_B(f)_\nu=S^{(0)}CE_B(P_\nu(f)),\end{equation}
where $(0)$ refers to both cases, Spherical expansion at a general $x$ or at $0$. Since
$$\|M_{sph,a}(f)-S^{(0)}CE_B(f)_\nu\|\le \|M_{sph,a}(f)-P_\nu(f)\| +\|P_\nu(f)-S^{(0)}CE_B(f)_\nu\|,$$
we can state our convergence theorems in the following form. For instance consider the example below. Taking into consideration that $\widehat{M_{sph,a}P_\nu(f)}=\eta\left(\frac{|\cdot|}{\nu}\right) M_{sph,a}(\hat{f})$, by Lemma \ref{leef} we have

\begin{theorem} Let $\nu>0$, $f\in L^1_a(\mathbb{R}^d_+)$ is of locally bounded variation and continuous. Assume that $\tilde{\phi}^sM_{sph,a}(\hat{f})^{(s)}(0,\cdot)\in L_1(\tilde{w})([0,\nu]$, then we have the following estimations.\\
If $f\in L^\infty$ with finite modulus of smoothness, we have
$$D_n:=\|M_{sph,a}(f)(0,\cdot)-V_n(S^0SE_B(f))_\nu\|_\infty$$ $$\le c \omega_{sph,m}\left(f, \frac{1}{\nu}\right)_{\infty,a}+\frac{c}{n^{s}}\|\tilde{\phi}^s M_{sph,a}(\hat{f})^{(s)}(0,\cdot)\|_{1,\tilde{w}}.$$
If $f\in W^{m,p}_{\Delta_a}$, we have
$$D_n\le  \frac{c}{\nu^{2m}}\|\Delta_a^m f\|_{\infty,a}+\frac{c}{n^{s}}\|\tilde{\phi}^s M_{sph,a}(\hat{f})^{(s)}(0,\cdot)\|_{1,\tilde{w}}.$$
\end{theorem}

\medskip

\section{Appendix}

\subsection{Shannon-type sampling}

Let $\mathrm{supp}\hat{f}\subset [0,1]$. The expansion $\hat{f}$ in terms of orthogonal polynomials leads to chromatic expansion of $f$, and expansion by $\{j_\alpha(\lambda_n x)\}_{n=1}^\infty$, leads to a Shannon-type sampling theorem.\\
Assume, that $\mathrm{supp}\hat{f}\subset [0,1]$. Let $\lambda_k=\lambda_k(\alpha)$ be the $k$th positive zero of $j_\alpha$. Then we have
$$T^{\lambda_k}f(x)=\frac{1}{2^{2\alpha}\Gamma^2(\alpha+1)}\int_0^1\hat{f}(y)j_\alpha(\lambda_ky)j_\alpha(xy)y^{2\alpha+1}dy$$
and by \eqref{j1}
$$T^{\lambda_k}f(0)=f(\lambda_k)=\frac{1}{2^{2\alpha}\Gamma^2(\alpha+1)}\int_0^1\hat{f}(y)j_\alpha(\lambda_ky)y^{2\alpha+1}dy.$$
Now let us taking into consideration that
$$\int_0^1j_\alpha(\lambda_nx)j_\alpha(\lambda_mx)x^{2\alpha+1}dx=\sigma_n^2\delta_{mn},$$
where
$$\sigma_n^2=\frac{1}{2}(j'_\alpha(\lambda_n))^2=\frac{\lambda_n^2}{8(\alpha+1)^2}(j_{\alpha+1}(\lambda_n))^2,$$
and $\{\frac{1}{\sigma_n}j_\alpha(\lambda_n \cdot)\}_{n=1}^\infty$ is a complete orthonormal system in $L^2_{\alpha, [0,1]}$, see e.g. \cite{st} and the references therein. Thus $c(\alpha)\frac{f(\lambda_k)}{\sigma_k}$ is the $k$th Fourier-Bessel coefficient of $\hat{f}$, i.e.
$$\hat{f}(y)=\sum_{k=1}^\infty\frac{1}{\sigma_k^2}\int_0^1\hat{f}(t)j_{\alpha}(\lambda_kt)t^{2\alpha+1}dtj_{\alpha}(\lambda_ky)\chi_{[0,1]}(y).$$
If it is integrable term by term, by the inversion formula we have
$$f(x)=\frac{1}{2^{2\alpha}\Gamma^2(\alpha+1)}\sum_{k=1}^\infty\int_0^1\hat{f}(t)j_{\alpha}(\lambda_kt)t^{2\alpha+1}dt\frac{1}{\sigma_k^2}\int_0^1j_{\alpha}(\lambda_ky)j_{\alpha}(xy)y^{2\alpha+1}dy.$$
In view of \cite[(2.2)]{cctv} (see also the references therein) we have
$$\frac{1}{\sigma_k^2}\int_0^1j_{\alpha}(\lambda_ky)j_{\alpha}(xy)y^{2\alpha+1}dy\frac{\lambda_k^2j_{\alpha+1}(\lambda_k)}{2(\alpha+1)\sigma_k^2}\frac{j_\alpha(x)}{\lambda_k^2-x^2}.$$
Finally, in accordance with \cite[(26)]{j} we have
$$f(x)=4(\alpha+1)j_\alpha(x)\sum_{k=1}^\infty \frac{f(\lambda_k)}{j_{\alpha+1}(\lambda_k)}\frac{1}{\lambda_k^2-x^2}.$$

Since $\lambda_k\sim k\pi$ and $j_{\alpha+1}(\lambda_k)\sim k^{-(\alpha+\frac{3}{2})}$ (see e.g. \cite{e} and \cite{sz}), its convergence properties are rather weak.

\medskip

\subsection{Numerical illustration}

We include a short numerical illustration of the compact-support Bessel-chromatic expansion. The purpose of this section is to visualize the convergence behavior of the truncated Bessel chromatic sums and of the corresponding de la Vall\'ee Poussin means. We consider the multidimensional weighted Bessel setting
\(
    n=3,
    \
    \alpha=\left(\frac12,1,\frac32\right).
\)
Then
\(
    a_i=2\alpha_i+1,
    \
    a=(2,3,4),
    \
    |a|=9.
\)
Therefore
\(
    n+|a|-1=11,
    \
    \gamma=\frac{n+|a|}{2}-1=5.
\)
Thus the radial weight appearing in the Hankel representation is
\(
    w(\rho)=\rho^{11},
\)
and the corresponding radial Bessel kernel is \(j_5(r\rho)\). We take \(\nu=3\) and consider the compactly supported filtered Gaussian spectral family
\[
    \widehat f_\beta(\rho)
    =
    A_\beta e^{-0.2\rho^2}
    \left(1-\frac{\rho}{3}\right)_+^\beta,
    \qquad
    \beta=\frac12,\frac32,\frac52.
\]
Here \(A_\beta\) is chosen so that
\(
    \|\widehat f_\beta\|_{L^2([0,3],\rho^{11}d\rho)}=1.
\)
The Gaussian factor gives a smooth spectral profile, while
\(
    \left(1-\frac{\rho}{3}\right)_+^\beta
\)
imposes compact support and controls the endpoint regularity at \(\rho=3\). The values of the normalization constants used in the computation
are
\[
    A_{1/2}=0.0603241961,
    \qquad
    A_{3/2}=0.2636216153,
    \qquad
    A_{5/2}=0.7976714616.
\]

We consider the compactly supported filtered Gaussian spectral family
\[
\widehat f_\beta(\rho)
=
A_\beta e^{-0.2\rho^2}
\left(1-\frac{\rho}{3}\right)_+^\beta,
\qquad
\beta=\frac12,\frac32,\frac52.
\]
The parameter \(\beta\) controls the endpoint regularity at \(\rho=3\).

\begin{figure}[htbp]
\centering
\includegraphics[width=0.72\textwidth]{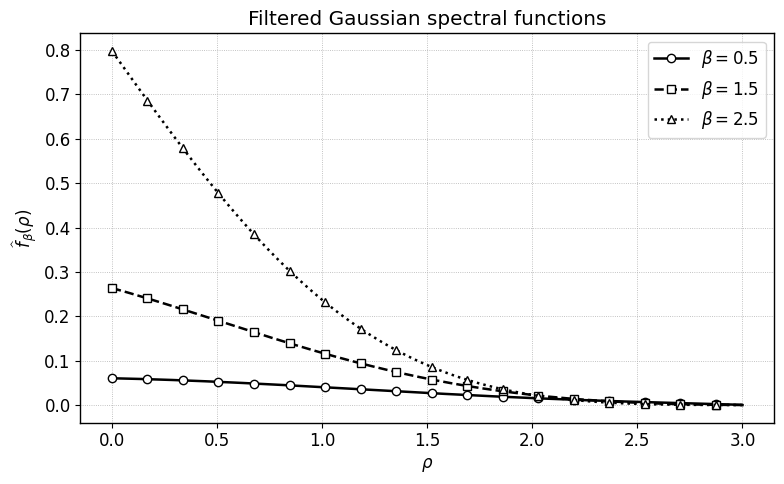}
\caption{Filtered Gaussian spectral functions for
\(\beta=1/2,3/2,5/2\).}
\label{fig:spectral-functions}
\end{figure}

The reference spherical mean is computed as
    $F_\beta(r)
    =
    \int_0^3
    \widehat f_\beta(\rho)j_5(r\rho)\rho^{11}\,d\rho .$

Let \(\{p_k\}_{k\ge0}\) be the orthonormal polynomial system on \([0,3]\)
with respect to the weight \(\rho^{11}\), that is,
    $\int_0^3 p_k(\rho)p_\ell(\rho)\rho^{11}\,d\rho=\delta_{k\ell}.$
The Bessel--chromatic coefficients are
$
    c_{k,\beta}
    =
    \int_0^3
    p_k(\rho)\widehat f_\beta(\rho)\rho^{11}\,d\rho,$
and the corresponding basis functions are
  $  \varphi_k(r)
    =
    \int_0^3
    p_k(\rho)j_5(r\rho)\rho^{11}\,d\rho .$
The \(N\)-th truncated reconstruction is
  $  S_{N,\beta}(r)
    =
    \sum_{k=0}^{N}c_{k,\beta}\varphi_k(r).$
We also compute the de la Vall\'ee Poussin mean $ V_{N,\beta}(r)
    =
    2\sigma_{2N,\beta}(r)-\sigma_{N,\beta}(r),$
where
 $\sigma_{N,\beta}(r)
    =
    \sum_{k=0}^{N-1}
    \left(1-\frac{k}{N}\right)c_{k,\beta}\varphi_k(r).$

All errors are measured in the relative uniform norm on \([0,10]\):
    $E^S_{N,\beta}
    =
    \frac{\|F_\beta-S_{N,\beta}\|_{L^\infty(0,10)}}
    {\|F_\beta\|_{L^\infty(0,10)}},$
and
   $ E^V_{N,\beta}
    =
    \frac{\|F_\beta-V_{N,\beta}\|_{L^\infty(0,10)}}
    {\|F_\beta\|_{L^\infty(0,10)}}.$
The use of relative errors is important because the functions \(F_\beta\)
have different amplitudes. In the computation,
\[
    \|F_{1/2}\|_\infty=163.3587791,
    \qquad
    \|F_{3/2}\|_\infty=95.95325259,
    \qquad
    \|F_{5/2}\|_\infty=58.24741283.
\]

\begin{table}[htbp]
\centering
\[
\begin{array}{c|ccc}
N & \beta=1/2 & \beta=3/2 & \beta=5/2\\
\hline
6  & 8.383\times10^{-7}  & 9.799\times10^{-8}  & 2.448\times10^{-8}\\
8  & 1.660\times10^{-7}  & 1.540\times10^{-8}  & 3.077\times10^{-9}\\
10 & 5.224\times10^{-8}  & 3.616\times10^{-9}  & 4.702\times10^{-10}\\
12 & 9.966\times10^{-9}  & 5.358\times10^{-10} & 5.465\times10^{-11}\\
16 & 1.530\times10^{-10} & 5.392\times10^{-12} & 3.581\times10^{-13}
\end{array}
\]
\caption{Relative \(L^\infty\)-errors
\(\|F_\beta-S_{N,\beta}\|_\infty/\|F_\beta\|_\infty\)
for the truncated Bessel--chromatic partial sums.}
\label{tab:partial-errors}
\end{table}

\begin{table}[htbp]
\centering
\[
\begin{array}{c|ccc}
N & \beta=1/2 & \beta=3/2 & \beta=5/2\\
\hline
6  & 1.274\times10^{-7}  & 1.556\times10^{-8}  & 3.313\times10^{-9}\\
8  & 1.925\times10^{-8}  & 1.818\times10^{-9}  & 3.720\times10^{-10}\\
10 & 4.398\times10^{-9}  & 3.121\times10^{-10} & 4.146\times10^{-11}\\
12 & 6.650\times10^{-10} & 3.671\times10^{-11} & 3.833\times10^{-12}\\
16 & 7.221\times10^{-12} & 2.611\times10^{-13} & 1.777\times10^{-14}
\end{array}
\]
\caption{Relative \(L^\infty\)-errors
\(\|F_\beta-V_{N,\beta}\|_\infty/\|F_\beta\|_\infty\)
for the de la Vall\'ee Poussin means.}
\label{tab:vp-errors}
\end{table}
\begin{figure}[htbp]
\centering
\includegraphics[width=0.72\textwidth]{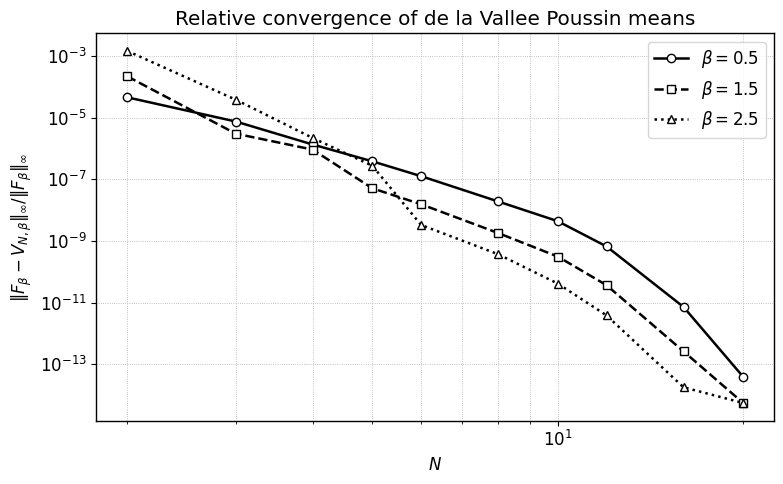}
\caption{Relative \(L^\infty\)-errors for the de la Vall\'ee Poussin
approximations \(V_{N,\beta}\). The decay becomes faster as the endpoint
regularity parameter \(\beta\) increases.}
\label{fig:vp-convergence}
\end{figure}

To quantify the decay, we estimate experimental convergence rates by fitting $ E_N\approx C N^{-s}$
on the log-log error curves over the range \(N=6,8,10,12,16\), before the
errors reach the level of numerical precision. This gives the approximate
rates shown in Table~\ref{tab:experimental-rates}.\\
\begin{table}[htbp]
\centering
\[
\begin{array}{c|cc}
\beta & S_{N,\beta} & V_{N,\beta}\\
\hline
1/2 & 8.478  & 9.695\\
3/2 & 9.717  & 10.947\\
5/2 & 11.113 & 12.181
\end{array}
\]
\caption{Experimental convergence rates obtained from log-log fits of the
relative errors over \(N=6,8,10,12,16\).}
\label{tab:experimental-rates}
\end{table}

The results show rapid convergence of both the ordinary partial sums and the
de la Vall\'ee Poussin means. After the initial low-order regime, the error
decreases faster as the endpoint regularity parameter \(\beta\) increases.
This is consistent with the approximation estimate above, where the convergence
rate depends on the weighted smoothness of the spectral side. The de la Vall\'ee
Poussin means give slightly smaller errors than the ordinary partial sums in
the tested range, which agrees with their role as a filtered approximation
procedure.

The coefficient decay gives the same indication. For larger \(\beta\), the
chromatic coefficients decay faster, showing that higher endpoint regularity
reduces the number of significant terms needed in the reconstruction. For
\(N\) close to \(20\), the errors are already near double-precision accuracy,
so increasing \(N\) further does not produce a visible improvement in this
computation.

The normalized coefficient decay is measured by
\(
\frac{|c_{k,\beta}|}{\|F_\beta\|_\infty}.
\)

\begin{figure}[htbp]
\centering
\includegraphics[width=0.72\textwidth]{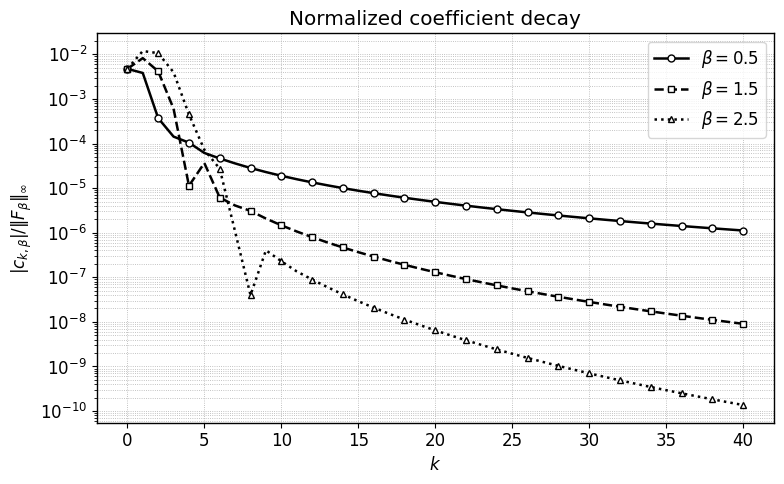}
\caption{Normalized decay of the Bessel--chromatic coefficients for
\(\beta=1/2,3/2,5/2\). Larger values of \(\beta\) produce faster coefficient
decay.}
\label{fig:coefficient-decay}
\end{figure}

\medskip

\vspace{5mm}

{\small{\noindent Both authors:\\
Department of Analysis and Operations Research,\\ Institute of Mathematics,\newline
Budapest University of Technology and Economics \newline
 M\H uegyetem rkp. 3., H-1111 Budapest, Hungary.

 \vspace{3mm}

\noindent \'A. P. Horv\'ath: g.horvath.agota@renyi.hu\\
M. Chegaar: mouna.chegaar@edu.bme.hu}


\medskip

\begin{thebibliography}{99}


\bibitem {be} Bateman H, Erd\'elyi A. Higher Transcendental Functions. New York:  McGraw-Hill; 1953. (vol.II).

\bibitem {bg} F. Bouzeffour, M. Garayev, On the fractional Bessel operator, {\it Integral Transforms and Special Functions} (2021), 230-246. DOI: 10.1080/10652469.2021.1925268

\bibitem {ch1} M. Chegaar, \'A. P. Horv\'ath,  Removable sets for fractional heat and fractional Bessel-heat equations. Integral Transforms and Special Functions, (2025) 1-20. https://doi.org/10.1080/10652469.2025.2567003

\bibitem {ch} M. Chegaar, \'A. P. Horv\'ath, Fractional Bessel-Sobolev and Bessel B-potential spaces, {\it J. of the Math. Soc. of Japan} (accepted)
https://doi.org/10.48550/arXiv.2509.03287


\bibitem {cdps} V. K. Chillara, E. S. Davis, C. Pantea, D. N. Sinha, Ultrasonic Bessel beam generation from radial modes of piezoelectric discs, {\it Ultrasonics}
{\bf 96} (2019) 140-148.

\bibitem {cctv} L. Colzani, A. Crespi, G. Travaglini, M. Vignati, Equiconvergence theorems for Fourier-Bessel expansions with application to harmonic analysis of radial functions in Euclidean and noneuclidean spaces, {\it Trans. of the Amer. Math. Soc.} {\bf 338} (1993) 43-55.

\bibitem {e} \'A. Elbert, Some recent results on the zeros of Bessel functions and orthogonal polynomials, {\it Journal of Computational and Applied Mathematics} {\bf 133} (2001) 65-83.

\bibitem {h1} \'A. P. Horv\'ath, Chromatic derivatives and expansions with weights, {\it Calcolo} {\bf 54} (2017) 1265 - 1291.

\bibitem {i0} A. Ignjatovi\'c, Numerical differentiation and signal processing, {\it Kromos Technology Technical Report}, Los Altos (2001)

\bibitem {i1} A. Ignjatovi\'c, Local approximations based on orthogonal differential operators, {\it The J. of Fourier Analysis and Appl.} {\bf 13} (3) (2007), 309-330.

\bibitem {i2} A. Ignjatovi\'c, Chromatic derivatives, chromatic expansions and associated spaces, {\it East J. Approx.} {\bf 15} (3) (2009), 263-302.

\bibitem {i3} A. Ignjatovi\'c, Frequency estimation using time domain methods based on robust differential operators, {\it 2010 IEEE 10th International Conference on Signal Processing (ICSP)} (2010), 151-154.

\bibitem {i4} A. Ignjatovi\'c, Chromatic derivatives, chromatic expansions and associated spaces II. (2025) https://arxiv.org/abs/2512.02326v1

\bibitem {iz} A. Ignjatovi\'c and A. I. Zayed, Multidimensional chromatic derivatives and series expansions, {\it Proc. of the Amer. Math. Soc.} {\bf 139} (10) (2011), 3513-3525.

\bibitem {j} A. J. Jerri, The Shannon sampling theorem - Its various extensions and applications, {\it Proc. of the IEEE} {\bf 65} (1977) 1565-1596.

\bibitem {k} H. P. Kramer, A generalized sampling theorem, {\it J. Math. Phys.} {\bf 38} (1959) 68-72.

\bibitem {lelu}  E. Levin, D. Lubinsky, Orthogonal polynomials for exponential weights $x^{2\varrho}e^{-Q(x)}$ on $[0,d)$, {\it Journal of Approximation Theory} {\bf 134} (2005), 199-256.

\bibitem {le} B.M. Levitan, Expansion in Fourier series and integrals with Bessel functions, (Russian) Uspehi Matem. Nauk (N.S.) 6(1951), No. 2(42),
102-143.

\bibitem{l} L. N. Lyakhov, A class of hypersingular integrals. {\it Dokl. Akad. Nauk SSSR} {\bf 315} (2) (1990) 291-296.
https://www.mathnet.ru/eng/dan6329

\bibitem{ms} G. Mastroianni, J. Szabados, Polynomial approximation on the real semiaxis with generalized Laguerre weights, {\it Studia Univ. Babes-Bolyai, Mathematica} {\bf LII} (2007), 105-128.

\bibitem{mave} G. Mastroianni, P. V\'ertesi, Some application of generalized Jacobi weights, {\it Acta Math. Hungar.} {\bf 77} (1997) 323-357.

\bibitem{mt} G. Mastroianni, V. Totik, Jackson type inequalities for doubling and $A_p$ weights. {\it Rend. Del Circolo Math. Di Palermo} Serie II Suppl. {\bf 52} 83-99 (1998)

\bibitem{mato} G. Mastroianni, V. Totik, Best bpproximation and moduli of smoothness for doubling weights, {\it Journal of Approximation Theory} {\bf 110} (2001) 180-199.

\bibitem{o} F. Oberhettinger, Tables of Bessel Transforms, Springer-Verlag, New York Heidelberg Berlin (1972).

\bibitem{p} S. S. Platonov, Bessel harmonic analysis and approximation of functions on the half-line, {\it Izv. RAN, Ser. Mat.},  {\bf 71} (2007), 149-196 (in Russian); translated in {\it Izv. Math.}, {\bf 71}:5 (2007), 1001-1048.

\bibitem {pbm} A.P. Prudnikov, Yu.A. Brychkov, O.I. Marichev, Integrals and series, Vol. 2, Special Functions. Gordon Breach Sci. Publ., New York,
1990.

\bibitem {s} E. Shishkina, Inversion of the weighted spherical mean, in Operator Theory and Harmonic Analysis, Springer Proceedings in Mathematics and Statistics, {\bf 357} 2021, 507-520.

\bibitem {esk} E. Shishkina, I. Ekincioglu, C. Keskin, Generalized Bessel potential and its application to non-homogeneous singular screened Poisson equation, {\it Integral Transforms and Special Functions} {\bf 32} (12) (2021) 932-947. https://doi.org/10.1080/10652469.2020.1867983

\bibitem {ss} E. Shishkina, S. Sitnik, {\it Transmutations, Singular and Fractional Differential Equations With Applications to Mathematical Physics}, Academic Press, London (2020).

\bibitem{st} K. Stempak, On convergence and divergence of Fourier-Bessel series, {\it Electronic Transactions on Numerical Analysis} {\bf 14} (2002) 223-235.

\bibitem{sz}  G. Szeg\H o, Orthogonal Polynomials, Amer. Math. Soc. New York, 1959.

\bibitem {x} Y. Xu, Approximation by polynomials in Sobolev spaces with Jacobi weight {\it J. Fourier Anal. Appl.} (2018) 1438-1459
https://doi.org/10.1007/s00041-017-9581-3

\bibitem {z1} A. I. Zayed, Generalizations of chromatic derivatives and series expansions, {\it IEEE Trans. on Signal Processing} {\bf 58} (3) (2010), 1638-1647.

\bibitem {z3} A. I. Zayed, Chromatic expansions of generalized functions, {\it Integral Transforms and Special Functions} {\bf 22} (4-5) (2011), 383-390.

\bibitem {z2} A. I. Zayed, Chromatic expansions in function spaces, {\it Trans. of the Amer. Math. Soc.} {\bf 366} (8) (2014), 4097-4125.




\end{thebibliography}
\end{document}